\documentclass[a4paper,fleqn]{cas-sc}

\usepackage[authoryear,longnamesfirst]{natbib}

\usepackage{amssymb,latexsym,amsfonts,amsthm,a4wide}              
\usepackage{mathrsfs}
\usepackage{mathtools} 
\usepackage{bm}%
\usepackage{float}
\usepackage{cases} 
\usepackage{textcomp} %
\usepackage{hyperref}
\usepackage{booktabs}%
\usepackage{color}
\usepackage{graphicx} 
\usepackage{subcaption} 

\newtheorem{theorem}{Theorem}[section]
\newtheorem{lemma}{Lemma}[section]

\numberwithin{equation}{section}

\newcommand{\normmm}[1]{{\left\vert\kern-0.25ex\left\vert\kern-0.25ex\left\vert #1   
   \right\vert\kern-0.25ex\right\vert\kern-0.25ex\right\vert}}                       

\def\tsc#1{\csdef{#1}{\textsc{\lowercase{#1}}\xspace}}
\tsc{WGM}
\tsc{QE}
\tsc{EP}
\tsc{PMS}
\tsc{BEC}
\tsc{DE}

\makeatletter
\@ifundefined{printorcid}{}{%
  \renewcommand{\printorcid}[1][]{}%
}
\makeatother

\begin{document}
\let\WriteBookmarks\relax
\def\floatpagepagefraction{1}
\def\textpagefraction{.001}
\shorttitle{Parallelizable Iterative Decoupled Algorithms for TP}
\shortauthors{Li Z.,Li J.}

\title [mode = title]{Parameters Robust Parallelizable Iterative Decoupled Algorithms for Thermo-Poroelasticity with Nonlinear Convective Transport}                      

%

\author[1]{Ziliang Li}[style=chinese]
\cormark[1]
\ead{li_zl@foxmail.com}

\affiliation[1]{organization={The College of Mathematical Sciences, Harbin Engineering University},
                postcode={150000}, 
                state={Harbin},
                country={China}}

\author[2]{Jingzhi Li}[style=chinese]
\ead{li.jz@sustech.edu.cn}

\affiliation[2]{organization={Department of Mathematics, Southern University of Science and Technology},
                postcode={518000}, 
                state={Shenzhen}, 
                country={China}}

\cortext[cor1]{Corresponding author}


\begin{abstract}
This paper presents several parallel iterative algorithms for the nonlinear thermo-poroelasticity. We introduce an auxiliary variable to transform the original model into a four-field formulation to overcome Poisson locking.  
The iterative idea is extended to linearize and decouple the four-field model. Different coupling strategies are employed among these three subproblems, combined with appropriately designed iteration settings, yielding distinct iterative parallel algorithms. Stabilization terms are incorporated into these algorithms to ensure parameter robustness in extreme situations. 
We derive error estimates and prove the convergence of these parallel iterative decoupled algorithms. Finally, numerical experiments are presented to compare the performance of the different parallel algorithms in extreme parameter settings, thereby validating their efficiency and robustness.
\end{abstract}




\maketitle

\section{Introduction}
In this study, we investigate numerical methods for the nonlinear thermo-poroelasticity model. This model describes the coupled interaction between non-isothermal fluid flow and the deformation of porous materials. It is a combination of Biot's equation \cite{biot1941general, terzaghi1943theoretical, biot1955theory} and the energy conservation equation. It specifically focuses on the interactions within poroelasticity, covering mechanical effects, fluid flow, and heat transfer in porous media.

Let $\Omega \subset \mathbb{R}^d$, with $d$ being either 2 or 3, denote an open and bounded region that has a Lipschitz polyhedral boundary $\partial \Omega$. Also, define $J=(0,\tau)$ as the time interval, where $\tau>0$ is the final time. The combined space-time region is then $\Omega\times J$.
We employ the variables  \(\bm u\), \(p\), and \(T\) to represent the displacement, fluid pressure, and temperature distribution, respectively, relative to a reference value. Define $\bm\varepsilon(\bm u)=\tfrac{1}{2}(\nabla\bm u+\nabla\bm u^T)$, and use $\bm I$ to denote the identity tensor. We apply the principles of momentum conservation, mass conservation, and energy conservation to derive three equations.
 Then, the general nonlinear thermo-poroelasticity model \cite{brun2018upscaling, brun2019well, brun2020monolithic} is presented in the following form: find $(\bm{u}, p, T)$ such that
\begin{equation}\label{eq: general TP_Model}
\begin{aligned} 
       -\nabla\cdot(2\mu\bm\varepsilon(\bm u)+\lambda\nabla\cdot\bm u\bm I)+\alpha\nabla p+\beta\nabla T
                      &=\bm f,\quad\text{in }\Omega\times J ,\\
\tfrac{\partial}{\partial t}(c_0 p - b_0 T+\alpha\nabla\cdot\bm u)-\nabla\cdot(\bm K\nabla p)
                      &=g,\quad\text{in }\Omega\times J ,\\
\tfrac{\partial}{\partial t}(a_0T - b_0p+\beta\nabla\cdot\bm u) 
-c_f(\bm K\nabla p)\cdot\nabla T
-\nabla\cdot(\bm\Theta\nabla T)
                      &=H_s,\quad\text{in }\Omega\times J,\\
\end{aligned}
\end{equation}
where the terms on the right-hand side, $\bm f$, $g$, and $H_s$, denote the body force, mass source, and heat source, respectively. 
For a comprehensive derivation of the model, readers can consult \cite{brun2018upscaling}.

The parameters $\mu$ and $\lambda$ are the Lamé parameters, and they can be written in terms of Young's modulus $E$ and Poisson's ratio $\nu$ as shown:
\(
\lambda=\tfrac{E\nu}{(1 + \nu)(1-2\nu)},\quad\mu=\tfrac{E}{2(1+\nu)}.
\)
The parameters $\bm K=(K_{ij})^d_{i,j = 1}$ and $\bm\Theta=(\Theta_{ij})^d_{i,j = 1}$ are matrices that are determined by the permeability of the medium and the viscosity of the fluid.
$a_0$, $b_0$, $c_0$, $\alpha$, and $\beta$ represent thermal capacity, thermal dilatation coefficient, specific storage coefficient, Biot--Willis constant, and thermal stress coefficient, respectively. $c_f$ is the fluid volumetric heat capacity divided by the reference temperature.
Detailed explanations of parameters, including their physical significance and corresponding units, as well as more comprehensive discussions, can be found in \cite{antonietti2023discontinuous,bonetti2024robust}.

To guarantee the existence and uniqueness of the solution, appropriate boundary and initial conditions must be specified. 
In this work, we consider mixed boundary conditions for $\bm{u}$, where the boundary $\partial \Omega$ is decomposed into the Dirichlet boundary $\Gamma_d$ and  Neumann boundary $\Gamma_n$ (i.e., $\partial\Omega=\Gamma_d\cup \Gamma_n$, $\Gamma_d\cap \Gamma_n=\varnothing$), with $\Gamma_d \neq \partial \Omega$. For pressure $p$ and temperature $T$, Dirichlet boundary conditions are prescribed. Without loss of generality, we assume all boundary conditions are homogeneous. Specifically,
\begin{equation}\label{eq: B_C}
\begin{aligned}
    \bm{u} = \bm 0, \quad& \text{on } \Gamma_d,\\
     ( 2\mu \bm{\varepsilon}(\bm{u}) + \lambda \nabla \cdot \bm{u} \bm{I} - \alpha  p\bm I - \beta  T \bm I)\bm n = \bm0 , \quad& \text{on } \Gamma_n   ,    \\
    p=T=0, \quad& \text{on } \partial\Omega.
\end{aligned}
\end{equation}
The initial conditions are presented as:
\begin{equation}\label{eq: I_C}
\bm{u}(\cdot, 0)=\bm{u}^0, \quad p(\cdot, 0)=p^0, \quad T(\cdot, 0)=T^0.
\end{equation}

As outlined in \cite{brun2020monolithic}, we make the following \textbf{Assumptions} for the model parameters throughout this paper (similar assumptions are also discussed in \cite{chen2022multiphysics,zhang2024coupling}):
 
(A1) $\bm{K}$ and $\bm\Theta$ are constant in time, symmetric, and positive definite. There exist $k_m>0$, $k_M>0$ and $\theta_m>0$, $\theta_M>0$ such that
$$ k_m|\bm v|^2 \le \bm v^T\bm K(x)\bm v,\quad |\bm K(x)\bm v|\le k_M|\bm v|,\quad\theta_m|\bm v|^2 \le \bm v^T\bm\Theta(x)\bm v,\quad |\bm\Theta(x)\bm v|\le\theta_M|\bm v|,\quad \forall \bm v\in\mathbb R^d . $$

(A2) The constants \(\lambda\), \(\mu\), \(\alpha\) and \(\beta\) are strictly positive constants.

(A3) The constants $a_0,b_0 $ and $c_0$ are such that $a_0,~ c_0 > b_0 \ge 0$ .

(A4) Impose the following regularity  :
\[ 
\bm{u} \in L^{\infty}(J;\bm{H}_{0}^{k+1}(\Omega)), \quad 
\partial_t \bm{u} \in L^{2}(J;\bm{H}_{0}^{k+1}(\Omega)), \quad 
\partial_{tt} \bm{u} \in L^{2}(J;\bm{H}_{0}^{1}(\Omega)),
\]
\[
\xi \in L^{\infty}(J;H^{s+1}(\Omega)), \quad 
\partial_t \xi \in L^{2}(J;H^{s+1}(\Omega)), \quad 
\partial_{tt} \xi \in L^{2}(J;L^{2}(\Omega)),
\]
\[
p \in L^{\infty}(J;H_{0}^{l+1}(\Omega)), \quad 
\partial_t p \in L^{2}(J;H_{0}^{l+1}(\Omega)), \quad 
\partial_{tt} p \in L^{2}(J;L^{2}(\Omega)),
\]
\[
T \in L^{\infty}(J;H_{0}^{l+1}(\Omega)), \quad 
\partial_t T \in L^{2}(J;H_{0}^{l+1}(\Omega)), \quad 
\partial_{tt} T \in L^{2}(J;L^{2}(\Omega)),
\]
where $k$, $s$, and $l$ are positive integers.

In addition to \textbf{Assumptions} (A1)-(A4), throughout the remainder of this paper, we assume that the initial conditions satisfy \(\bm{u}^0 \in [H^1(\Omega)]^d\) and \(p^0, T^0 \in L^2(\Omega)\), while the source terms satisfy \(\bm{f} \in [L^2(\Omega)]^d\), \(g \in L^2(\Omega)\), and \( H_{s} \in L^2(\Omega)\).

The well-posedness of the thermo-poroelastic model has been established in studies such as \cite{brun2019well}. For the numerical solution of these models, the most direct approach involves coupled algorithms, which solve the fully coupled system of governing equations simultaneously, offering the benefit of unconditional stability. For instance, the standard Galerkin method for the thermo-poroelasticity model problem is proposed in \cite{zhang2022galerkin}, where the corresponding error analysis is also derived. Other works, such as \cite{zhang2024coupling} and \cite{Zhang2023mfe}, have combined mixed finite element, Galerkin element, and characteristic finite element methods, respectively. Furthermore, robust discontinuous Galerkin methods for fully coupled nonlinear models are detailed in \cite{bonetti2024robust, antonietti2023discontinuous}. 
\cite{cai2026parameter} proposes two types of preconditioners that are robust with respect to variations in parameters and mesh refinement for linear thermo-poroelastic models to improve computational efficiency. 
For the decoupling methods, \cite{chen2022multiphysics} proposes a multiphysics finite element method for the quasi-static thermo-poroelasticity model by reformulating the original problem. A further option involves iterative algorithms, in which the solution for each variable is updated sequentially using solutions from previous iterations. While iterative algorithms have been widely applied to poroelastic models, their application to the thermo-poroelastic model remains relatively limited. In this context, \cite{brun2020monolithic} extends the original three-field thermo-poroelastic model to a five-field formulation to study the problem under different coupling scenarios.  An efficient iterative decoupling method based on a four-field formulation was presented in \cite{cai2025efficient}, and a decoupled iterative solution approach combined with reduced-order modeling was developed to enhance computational efficiency, as in \cite{Ballarin2024projection}.

Numerically approximating the nonlinear thermo-poroelasticity system poses notable computational challenges, partly due to the wide range of spatially and temporally varying model parameters encountered in practical engineering applications. The main contribution of this work is the development of several robust parallel iterative algorithms tailored for nonlinear thermo-poroelasticity problems, which exhibit consistent parameter robustness across the full spectrum of realistic parameter variations.
To address the challenge of Lamé parameters approaching infinity, we introduce an auxiliary variable that transforms the original three-field model into a four-field formulation.
In Section \ref{sec: coupled algorithm}, we investigate two ill-conditioned scenarios from an algebraic perspective, which arise when adopting either the three-field model or the four-field model. Accordingly, dedicated stabilization terms are incorporated into the proposed algorithms to effectively suppress the numerical instability of pressure and temperature under extreme parameter conditions.
For parallelization, the system is split into a mixed linear elasticity (M) subproblem, a flow subproblem (F), and a heat subproblem (H).
Different coupling strategies of these three subproblems are employed to yield distinct iterative algorithms.
With appropriately designed iteration settings, the decoupling algorithms(M-FH, MF-H, MH-F, and M-F-H) can be executed in parallel.
We approximate the nonlinear convection term by fully replacing it at the current iteration step with its value from the previous iteration step. Other linearization methods, such as a standard fixed-point linearization strategy or Newton-type approaches, can also be used, but these are beyond the scope of this paper.
An error estimate is derived for the non-iterative discretized nonlinear model. The convergence of all iterative algorithms to the solution of this discretized model is also established.
Consequently, the parallel iterative algorithms maintain optimal convergence rates, providing a reliable and accurate computational framework for thermo-poroelastic processes.
It is noteworthy that all iterative decoupled algorithms converge to the nonlinear coupled solution independently of physical parameters.

The paper is structured as follows.
Section \ref{sec: PDE model} introduces a generalized nonlinear thermo-poroelastic model along with its four-field formulation.
Section \ref{sec: coupled algorithm} presents an iterative coupling algorithm and several parallel iterative decoupling algorithms.
Section \ref{sec: the error estimate and convergence}  presents an error analysis of the nonlinear coupled algorithm and conducts a convergence analysis of the iterative decoupling algorithms.   
Section \ref{sec: experiments} offers numerical experiments to validate the theoretical findings.

\section{Poisson Locking-Free Four-Field Formulation}\label{sec: PDE model}

As in Biot's model, the displacement $\bm u$ will get into a divergence-free state in thermo-poroelasticity as the Lamé constant $\lambda\to\infty$ \cite{boffi2013mixed,yi2017study}, which is called Poisson locking.  
To eliminate Poisson locking, we now introduce the four-field formulation for the above model \eqref{eq: general TP_Model}. More clearly, following the approach presented in literature in \cite{lee2017parameter,oyarzua2016locking,cai2025efficient}, we define an auxiliary variable:  
$$  
\xi = -\lambda \nabla \cdot \bm{u} + \alpha p + \beta T,
$$  
which is commonly referred to as the pseudo-total pressure \cite{antonietti2023discontinuous}. Substituting this variable into equation \eqref{eq: general TP_Model}, the four-field thermo-poroelasticity problem is as follows.

\begin{equation}\label{TP_Model_four}
\begin{aligned} 
      -2\mu \nabla\cdot \bm\varepsilon(\bm u) +\nabla\xi 
                      =\bm f &~~~(\bm M) , \\
         -\lambda \nabla\cdot\bm u-\xi+\alpha p+\beta T
                      =0 &, \\
 -\tfrac\alpha\lambda\tfrac\partial{\partial t}\xi+(c_0+\tfrac{\alpha^2}\lambda)\tfrac\partial{\partial t}p
                 +(\tfrac{\alpha\beta}\lambda-b_0)\tfrac\partial{\partial t}T
                      -\nabla\cdot(\bm K\nabla p)=g &~~~(\bm F), \\
 -\tfrac\beta\lambda\tfrac\partial{\partial t}\xi +(\tfrac{\alpha\beta}\lambda-b_0)\tfrac\partial{\partial t}p
+(a_0+\tfrac{\beta^2}\lambda)\tfrac\partial{\partial t}T
              -c_f(\bm K\nabla p)\cdot\nabla T    -\nabla\cdot(\bm \Theta\nabla T)
                      = H_{s} &~~~(\bm H). 
\end{aligned}
\end{equation}
We still use the initial condition (\ref{eq: I_C}) and the Dirichlet boundary condition for \eqref{TP_Model_four}. In practical situations, nonhomogeneous Dirichlet and Neumann boundary conditions are commonly encountered. The analysis performed for homogeneous boundary conditions can be straightforwardly extended to accommodate these cases.

Define the following spaces:
$$
\bm V=[ H^1_{0,\Gamma_d}(\Omega) ]^d,~Q=L^2(\Omega),~W=H_0^1(\Omega). 
$$

In the remaining part of this paper, for convenience, we introduce the following parameter transformations:  
\[
c_\alpha = c_0 + \tfrac{\alpha^2}{\lambda}, ~ c_{\alpha\beta} = - b_0+\tfrac{\alpha \beta}{\lambda} , ~ c_\beta = a_0 + \tfrac{\beta^2}{\lambda}.
\]  

Applying parameter transformations, the variational formulation for the four-field linear thermo-poroelasticity problem reads as: for almost every $t\in (0, \tau]$, find $(\bm u,~\xi,~p,~T)\in\bm V\times Q\times W\times W$ such that
\begin{equation}\label{TP_Model_var after parameters transformation}
\begin{aligned} 
2\mu(\bm\varepsilon(\bm u), \bm\varepsilon(\bm v)) - (\nabla\cdot\bm v,\xi)&=(\bm f,\bm v)       ,~\forall \bm v\in\bm V,  \\
  (\nabla\cdot\bm u,\phi)+\tfrac{1}{\lambda}(\xi,\phi)-\tfrac{\alpha}{\lambda}(p,\phi)-\tfrac{\beta}{\lambda}(T,\phi)&=0 ,~\forall\phi\in Q,\\
 -\tfrac\alpha\lambda( \xi_t ,q)+c_\alpha( p_t ,q) + c_{\alpha\beta}( T_t ,q)
  +(\bm K\nabla p,\nabla q)&=(g,q),~\forall q\in W,\\
 -\tfrac\beta\lambda( \xi_t ,S)+c_{\alpha\beta}( p_t ,S)+c_\beta( T_t ,S)   
 -c_f(\bm K\nabla p\cdot\nabla T,S)&\\
  +(\bm\Theta\nabla T,\nabla S)&=( H_{s},S), ~\forall S\in W.\\
\end{aligned}
\end{equation}  
 For this problem, the initial condition for $\xi$ is
\begin{equation}\label{eq: initial xi0}
 \xi(\cdot,0) = \xi^0 = -\lambda \nabla \cdot \bm{u}^0 + \alpha p^0 + \beta T^0. 
\end{equation}
Moreover, the initial conditions for other variables, \eqref{eq: I_C}, \eqref{eq: initial xi0}, and the Dirichlet boundary condition are also applied to problem \eqref{TP_Model_var after parameters transformation}.
 
Let $h$ denote the maximum diameter of the mesh elements, and for $h > 0$, let $\mathcal{T}_h$ be a family of triangulations of the domain $\Omega$ consisting of triangular elements. The triangulations are assumed to be shape-regular and quasi-uniform. The discrete approximation spaces for $\bm{u}$, $\xi$, $p$ and $T$ are denoted by $\bm{V}_h \times Q_h\times W_h\times W_h$.

For time discretization, we consider an equidistant partition consisting of a series of points from $0$ to $\tau$ with a constant step size of $\Delta t=\tau/N$, where $N$ is a positive integer. Specifically, we define $\bm u^n$ as $\bm u(t_n)$, where $t_n$ represents the $n$-th point in the partition. Specifically, we define \(\bm{u}^n = \bm{u}(t_n)\), where \(t_n\) is the \(n\)-th time level. Similarly, \(\xi^n = \xi(t_n)\), \(p^n = p(t_n)\), \(T^n = T(t_n)\), \(\bm{f}^n = \bm{f}(t_n)\), \(g^n = g(t_n)\), and \(H_s^n = H_s(t_n)\).
The four-field formulation \eqref{TP_Model_var after parameters transformation}, when discretized in time using the backward Euler method, produces the following fully coupled algorithm.

Suppose that an initial value $(\bm u_h^0,~\xi_h^0,~p_h^0,~T_h^0) =(\bm u^0,~\xi^0,~p^0,~T^0)$ is provided.
For the given finite element spaces $\bm V_h \subset \bm V$, $Q_h\subset Q$, $W_h\subset W $, the coupled algorithm reads as, find $(\bm u_h^n,~\xi_h^n,~p_h^n,~T_h^n)\in\bm V_h\times Q_h\times W_h\times W_h$
such that for all $(\bm v_h ,~\phi_h ,~q_h ,~S_h )\in\bm V_h\times Q_h\times W_h\times W_h$,
\begin{subequations}\label{eq:coupled alg}
\begin{align}   \label{TP_Model_dis_a}
2\mu( \bm\varepsilon(\bm u^n_h),\bm\varepsilon(\bm v_h)) - (\nabla\cdot\bm v_h,\xi^n_h)=&(\bm f^n,\bm v_h),  
\\    \label{TP_Model_dis_b}
  (\nabla\cdot\bm u^n_h,\phi_h)+\tfrac{1}{\lambda}(\xi^n_h,\phi_h)-\tfrac{\alpha}{\lambda}(p^n_h,\phi_h)-\tfrac{\beta}{\lambda}(T^n_h,\phi_h)=&0 ,
\\ \label{TP_Model_dis_c}
 -  \tfrac\alpha\lambda(\xi_h^n,q_h)+ c_\alpha(p_h^n,~q_h)
 + c_{\alpha\beta}(T_h^n,~q_h) 
                +   \Delta t( \bm K\nabla  p^n_h,~\nabla q_h)=&(\tilde g,q_h), \\  
     \begin{split}\label{TP_Model_dis_d}
  -\tfrac\beta\lambda(\xi_h^n,S_h)+c_\beta(T_h^n,S_h)
  +c_{\alpha\beta}(p_h^n,~S_h)   
           +\Delta t(\bm\Theta\nabla T^n_h,~\nabla S_h) &  \\
-\Delta tc_f( \bm{K}\nabla p_h^{n} \cdot \nabla T_h^{n}, S_h)    = &    
       (\tilde H_s,S_h)   , 
	\end{split}
\end{align}
\end{subequations}
where 
\begin{equation} \label{eq:gn qh}
(\tilde{g}, q_h) = c_\alpha(p_h^{n-1}, q_h) + c_{\alpha\beta}(T_h^{n-1}, q_h) - \tfrac{\alpha}{\lambda}(\xi_h^{n-1}, q_h) + \Delta t(g^n, q_h),
\end{equation}
\begin{equation} \label{eq:Hsn Sh}
(\tilde{H}_s, S_h) = c_\beta(T_h^{n-1}, S_h) + c_{\alpha\beta}(p_h^{n-1}, S_h) - \tfrac{\beta}{\lambda}(\xi_h^{n-1}, S_h) + \Delta t(H_s^n, S_h).
\end{equation}

‌We present the following theorem to demonstrate that the proposed discrete nonlinear algorithm \eqref{eq:coupled alg} is free of Poisson locking as $\lambda\to\infty$. The proof of the theorem is presented in Section \ref{sec: the error estimate and convergence}.
\begin{theorem} \label{th: error estimate}   
 Assume that (A1)-(A4) hold true and $\bm V_h,Q_h$ satisfy the discrete inf-sup condition. 
Let \((\bm u,~\xi,~p,~T)\) and \((\bm u_h^{n},~\xi_h^{n},~p_h^{n},~T_h^{n})\) for \(n\ge0\) be the solutions of equations\eqref{TP_Model_var after parameters transformation} and equations\eqref{TP_Model_dis_a}-\eqref{TP_Model_dis_d}, respectively. There holds    
\begin{equation} \label{eq: non error estimate}     
\begin{aligned}     
& 
2\mu\Vert\bm\varepsilon(\bm u-\bm u_h^N)\Vert_{L^2(\Omega)}^2         
+\Vert \xi-\xi_h^N\Vert_{L^2(\Omega)}^2       
+(c_0-b_0) \Vert p-p_h^N\Vert_{L^2(\Omega)}^2         
+(a_0-b_0)\Vert T-T_h^N\Vert_{L^2(\Omega)}^2                       
\\
&+ \tfrac12 \Delta t \sum_{n=0}^{N-1} \big(              
 k_m \|\nabla p-\nabla p_h^n \|^2_{L^2(\Omega)}
    + \theta_m \|\nabla T-\nabla T_h^n \|^2_{L^2(\Omega)} \big)\\    
\le &  \tilde C_1\Delta t^2 + \tilde C_2 h^{2k} + \tilde C_3 h^{2s+2}
                            + \tilde C_4 h^{2l+2},      
\end{aligned}    
\end{equation}      
where     
\begin{equation*}      
\begin{aligned}     
&  \tilde C_1   =\tilde C\Big(  \int_{0}^{\tau}\| \bm u_{tt}  \|_{H^1(\Omega)}^2                   
+\int_{0}^{\tau}\| \xi_{tt}  \|_{L^2(\Omega)}^2                   
+\int_{0}^{\tau}\| p_{tt}  \|_{L^2(\Omega)}^2                 
  +\int_{0}^{\tau}\| T_{tt}  \|_{L^2(\Omega)}^2 
   \Big),\\     
&  \tilde C_2   =\tilde C           
       \int_{0}^{\tau}\| \bm u_t  \|_{H^{k+1}(\Omega)}^2 , \\      
&  \tilde C_3   =\tilde C                
            \int_{0}^{\tau}\| \xi_t  \|_{H^{s+1}(\Omega)}^2  ,   \\
&  \tilde C_4   =\tilde C \left(               
\int_{0}^{\tau}\| p_t  \|_{H^{l+1}(\Omega)}^2                
+  \int_{0}^{\tau}\| T_t  \|_{H^{l+1}(\Omega)}^2\right), 
\end{aligned}    
\end{equation*} 
where  \(\tilde C=\tilde C(1/\lambda,\alpha,\beta,a_0,b_0,c_0,c_f,k_m,k_M,\theta_m,\theta_M)\).   
\end{theorem}

\section{Parallel Iterative Algorithms with Stabilization Terms}\label{sec: coupled algorithm}
Numerical instability arises when the parameters take extreme values.
Define $\bm\sigma(\bm u)=2\mu\bm\varepsilon\bm u+\lambda tr(\bm\varepsilon\bm u)\bm I$.
The operator form of the system \eqref{eq: general TP_Model} can be expressed as
\begin{equation*} 
\small
\begin{bmatrix}
\bm\sigma     &\alpha I             &  \beta I                                \\   
\alpha\tfrac{\partial}{\partial t}\nabla\cdot& 
c_0\tfrac{\partial}{\partial t}I-\nabla\cdot\bm K\nabla    
                                           &-b_0\tfrac{\partial}{\partial t}I\\
\beta\tfrac{\partial}{\partial t}\nabla\cdot&
-b_0\tfrac{\partial}{\partial t}I    
                                           & 
a_0\tfrac{\partial}{\partial t}I-\nabla\cdot\bm\Theta\nabla-c_f(\bm K\nabla p)\cdot\nabla\\
\end{bmatrix}
\begin{bmatrix}
 \bm u \\ p \\ T
\end{bmatrix}
=
\begin{bmatrix}
 \bm f \\   g \\ H_s
\end{bmatrix}.
\end{equation*}

When the material parameters $a_0$, $b_0$, and $c_0$ take small or zero values, while the sufficiently small permeability tensor $\bm{K}$ and thermal conductivity tensor $\bm{\Theta}$ render the diffusive and coupling terms $c_f(\bm{K}\nabla p)\cdot\nabla T$, $\nabla\cdot(\bm{K}\nabla p)$, and $\nabla\cdot(\bm{\Theta}\nabla T)$ asymptotically negligible, the lower-right $2\times2$ block is approximately a zero matrix relative to the remaining parameters, which makes the coefficient matrix ill-posed. On the other hand, when $c_f$ is sufficiently large compared to other parameters, the nonlinear convection term becomes dominant, and $c_f(\bm{K}\nabla p)\cdot\nabla T$ introduces a nonsymmetric perturbation that significantly increases the condition number of the coefficient matrix, further leading to severe ill-posedness.

Upon switching to the four-field formulation, the $2\times2$ lower-right sub-block of the system remains unaltered. As a direct consequence, the ill-posedness persists across the four-field model.

To circumvent this ill-posedness, one regularization strategy is to introduce a diagonal stabilization term, which restores the invertibility of the algebraic system. At the same time, these stabilization terms also endow the algorithms with guaranteed convergence.

Next, we present several iterative algorithms equipped with stabilization terms $\rho_p(p^{n, i}_h-p^{n,i-1}_h,  q_h)$ and $\rho_T(T^{n, i}_h-T^{n,i-1}_h,  S_h)$, which integrate distinct decoupling strategies to suppress the numerical instability in pressure and temperature fields.
An additional $\rho_\xi(\xi^{n, i}_h-\xi^{n,i-1}_h,  \phi_h)$ term is also needed to ensure convergence for the MF-H and the MH-F algorithms.

\subsection{The parallel iterative algorithm}\label{sub: The parallel iteratively decoupled algorithm}
In the next parts, let $0\le \rho_p$, $\rho_T$, $\rho_\xi\in\mathbb R$.
Let \(i\) be the iteration index, and we propose the iterative algorithm for solving \((\bm{u}_h^n, \xi_h^n, p_h^n, T_h^n)\) :
For each time step \(t_n\) (\(n \geq 1\)), let the previous iteration step states \((p_h^{n, i-1}, T_h^{n, i-1}, \xi_h^{n, i-1})\) be given, with the initial iteration values set as \(p_h^{n,0} = p_h^{n-1}\), \(T_h^{n,0} = T_h^{n-1}\), and \(\xi_h^{n,0} = \xi_h^{n-1}\).   
We first present the following iteratively fully coupled algorithm.
\subsubsection{MFH algorithm:}

 Given $p_h^{n,{i-1}}$, $T_h^{n,{i-1}}$, $\xi_h^{n,i-1}$, find $(\bm u^{n,i}_h,~\xi^{n,i}_h,~p_h^{n,i},~T_h^{n,i})$ such that for all $(\bm v_h ,~\phi_h ,~q_h ,~S_h )\in\bm V_h\times Q_h\times W_h\times W_h$,
\begin{equation}\label{alg: MFH} 
  \begin{aligned}  
    2\mu(\bm\varepsilon(\bm u^{n,i}_h), \bm\varepsilon(\bm v_h))
  - (\nabla\cdot\bm v_h,\xi^{n,i}_h) 
                &=(\bm f^n,\bm v_h),  \\
   (\nabla\cdot\bm u^{n,i}_h,\phi_h)+\tfrac{1}{\lambda}(\xi^{n,i}_h,\phi_h)   
  -\tfrac{\alpha}{\lambda}(p^{n,{i}}_h,\phi_h)-\tfrac{\beta}{\lambda}(T^{n,{i}}_h,\phi_h)
                &=  0,  \\ 
      -  \tfrac\alpha\lambda(\xi_{h}^{n,i},q_h)+  c_\alpha(p_{h}^{n,i},q_h) 
               + c_{\alpha\beta}(T_{h}^{n,i},q_h) 
                + \Delta t(\bm K\nabla p^{n,i}_h,\nabla q_h)&\\
                    + \rho_p(p^{n,i}_h-p^{n,i-1}_h,  q_h) &  =   (\tilde{g}, q_h)       ,\\
     -\tfrac\beta\lambda(\xi_h^{n,i},S_h)+c_\beta(T_h^{n,i},S_h)
              +c_{\alpha\beta}(p_h^{n,i},S_h) 
             +\Delta t(\bm\Theta\nabla T^{n,i}_h,\nabla S_h)&\\ 
          +  \rho_T(T^{n,i}_h-T^{n,i-1}_h,  S_h)    =  (\tilde{H}_s, S_h)  
            +\Delta t(c_f\bm{K}\nabla p_h^{n,i-1}\cdot& \nabla T_h^{n,i-1}, S_h),\\
\end{aligned}
\end{equation}
where $(\tilde{g}, q_h)$ and $(\tilde{H}_s, S_h)$ here as well as in the following text are defined in \eqref{eq:gn qh}-\eqref{eq:Hsn Sh}. In the algorithm, the whole nonlinear term is replaced by the value of the last iteration step to ensure symmetry and simplicity.  Other linearization methods, such as a standard fixed-point linearization strategy or Newton-type approaches \cite{huang2016newton}, can also be used.

Instead of solving the system \eqref{TP_Model_dis_a}-\eqref{TP_Model_dis_d} in a fully coupled way, we can reformulate it into two or three decoupled subproblems. We apply different splitting strategies to decouple the governing equations, yielding distinct iterative algorithms as presented below.  
 
\subsubsection{Parallel M-FH algorithm}
 Given $p_h^{n,{i-1}},T_h^{n,{i-1}}$,find $(\bm u_h^{n,i},\xi_h^{n,i})$ such that for all $(\bm v_h ,~\phi_h)\in\bm V_h\times Q_h$,
\begin{equation}\label{alg: M FH uxi} 
\begin{aligned}  
    2\mu(\bm\varepsilon(\bm u^{n,i}_h), \bm\varepsilon(\bm v_h))
  - (\nabla\cdot\bm v_h,\xi^{n,i}_h) 
                &=(\bm f^n,\bm v_h),  \\
   (\nabla\cdot\bm u^{n,i}_h,\phi_h)+\tfrac{1}{\lambda}(\xi^{n,i}_h,\phi_h)
                &=  \tfrac{\alpha}{\lambda}(p^{n,{i-1}}_h,\phi_h)+\tfrac{\beta}{\lambda}(T^{n,{i-1}}_h,\phi_h). \\ 
\end{aligned}
\end{equation}
 Given $p_h^{n,{i-1}},T_h^{n,{i-1}},\xi_h^{n,i-1}$, find $(p_h^{n,i},T_h^{n,i})$ such that for all $(q_h ,~S_h )\in W_h\times W_h$,
\begin{equation}\label{alg: M FH pT} 
  \begin{aligned}  
       c_\alpha(p_{h}^{n,i},q_h) 
            &   + c_{\alpha\beta}(T_{h}^{n,i},q_h) 
                + \Delta t(\bm K\nabla p^{n,i}_h,\nabla q_h)
                + \rho_p(p^{n,i}_h-p^{n,i-1}_h,  q_h)
                  \\
                =&    (\tilde g,q_h)     
                +  \tfrac\alpha\lambda(\xi_{h}^{n,i-1},q_h)  ,\\
      c_\beta(T_h^{n,i},S_h)
            & +c_{\alpha\beta}(p_h^{n,i},S_h) 
              +\Delta t(\bm\Theta\nabla T^{n,i}_h,\nabla S_h)
              + \rho_T(T^{n,i}_h-T^{n,i-1}_h,  S_h)\\ 
               =&  ( \tilde H_{s} ,S_h)  +\tfrac\beta\lambda(\xi_h^{n,i-1},S_h)
                +\Delta t(c_f\bm{K}\nabla p_h^{n,i-1}\cdot \nabla T_h^{n,i-1}, S_h).\\
\end{aligned}  
\end{equation}

\subsubsection{Parallel MF-H algorithm:} 
 Given $p_h^{n,{i-1}}$, $T_h^{n,{i-1}}$, find $(\bm u^{n,i}_h,~\xi^{n,i}_h,~p_h^{n,i})$ such that for all $(\bm v_h ,~\phi_h ,~q_h)\in\bm V_h\times Q_h\times W_h$,
\begin{equation}\label{alg: MF_H 1} 
  \begin{aligned}  
    2\mu(\bm\varepsilon(\bm u^{n,i}_h), \bm\varepsilon(\bm v_h))
  - (\nabla\cdot\bm v_h,\xi^{n,i}_h) 
                &=(\bm f^n,\bm v_h),  \\
   (\nabla\cdot\bm u^{n,i}_h,\phi_h)+\tfrac{1}{\lambda}(\xi^{n,i}_h,\phi_h)   
  -\tfrac{\alpha}{\lambda}(p^{n,{i}}_h,\phi_h)
  +\rho_\xi(\xi^{n, i}_h-\xi^{n,i-1}_h,  \phi_h)
                &=  \tfrac{\beta}{\lambda}(T^{n,i-1}_h,\phi_h),  \\ 
      -  \tfrac\alpha\lambda(\xi_{h}^{n,i},q_h)+  c_\alpha(p_{h}^{n,i},q_h) 
                + \Delta t(\bm K\nabla p^{n,i}_h,\nabla q_h) + \rho_p(p^{n,i}_h-p^{n,i-1}_h,  q_h)&  = \\
- c_{\alpha\beta}(T_{h}^{n,i-1},q_h) &+  (\tilde{g}, q_h)       .\\
\end{aligned}
\end{equation}
 Given $p_h^{n,{i-1}}$, $T_h^{n,{i-1}}$, $\xi_h^{n,i-1}$, find $T_h^{n,i}$ such that for all $S_h \in W_h$,
\begin{equation}\label{alg: MF_H 2} 
  \begin{aligned}  
  c_\beta(T_h^{n,i},S_h)
             +\Delta t(\bm\Theta\nabla T^{n,i}_h,\nabla S_h) 
          + \rho_T(T^{n,i}_h-T^{n,i-1}_h,  S_h) = \tfrac\beta\lambda(\xi_h^{n,i-1},S_h)
& \\
  -c_{\alpha\beta}(p_h^{n,i-1},S_h) + (\tilde{H}_s, S_h)   +\Delta t(c_f\bm{K}\nabla p_h^{n,i-1}\cdot \nabla T_h^{n,i-1}, S_h)&.\\
\end{aligned}
\end{equation}

\subsubsection{Parallel MH-F algorithm:}  
 Given $p_h^{n,{i-1}}$, $T_h^{n,{i-1}}$, find $(\bm u^{n,i}_h,~\xi^{n,i}_h,~T_h^{n,i})$ such that for all $(\bm v_h ,~\phi_h,~S_h )\in\bm V_h\times Q_h\times W_h$,
\begin{equation}\label{alg: F_HM 1} 
  \begin{aligned}  
    2\mu(\bm\varepsilon(\bm u^{n,i}_h), \bm\varepsilon(\bm v_h))
  - (\nabla\cdot\bm v_h,\xi^{n,i}_h) 
                &=(\bm f^n,\bm v_h),  \\
  (\nabla\cdot\bm u^{n,i}_h,\phi_h)+\tfrac{1}{\lambda}(\xi^{n,i}_h,\phi_h)   
  -\tfrac{\beta}{\lambda}(T^{n,{i}}_h,\phi_h)
  +\rho_\xi(\xi^{n, i}_h-\xi^{n,i-1}_h,  \phi_h)
                &=  \tfrac{\alpha}{\lambda}(p^{n,i-1}_h,\phi_h),  \\ 
     -\tfrac\beta\lambda(\xi_h^{n,i},S_h)+c_\beta(T_h^{n,i},S_h)
             +\Delta t(\bm\Theta\nabla T^{n,i}_h,\nabla S_h) + \rho_T(T^{n,i}_h-T^{n,i-1}_h,  S_h)
               &= \\
-c_{\alpha\beta}(p_h^{n,i-1},S_h)  + (\tilde{H}_s, S_h)  
            +\Delta t(c_f\bm{K}\nabla p_h^{n,i-1}\cdot& \nabla T_h^{n,i-1}, S_h).\\
\end{aligned}
\end{equation}
 Given $T_h^{n,{i-1}}$, $\xi_h^{n,i-1}$, find $p_h^{n,i}$ such that for all $S_h \in W_h$,
\begin{equation}\label{alg: F_HM 2} 
  \begin{aligned}   
      c_\alpha(p_{h}^{n,i},q_h)      
  + \Delta t(\bm K\nabla p^{n,i}_h,\nabla q_h) + \rho_p(p^{n,i}_h-p^{n,i-1}_h,  q_h)&  = \\
  \tfrac\alpha\lambda(\xi_{h}^{n,i-1},q_h) - c_{\alpha\beta}(T_{h}^{n,i-1},q_h) & +  (\tilde{g}, q_h).\\
\end{aligned}
\end{equation}

\subsubsection{Parallel M-F-H algorithm:}  
 Given $p_h^{n,{i-1}}$, $T_h^{n,{i-1}}$, find $(\bm u^{n,i}_h,~\xi^{n,i}_h)$ such that for all $(\bm v_h ,~\phi_h)\in\bm V_h\times Q_h$,
\begin{equation}\label{alg: H_F_M 1} 
  \begin{aligned}  
    2\mu(\bm\varepsilon(\bm u^{n,i}_h), \bm\varepsilon(\bm v_h))- (\nabla\cdot\bm v_h,\xi^{n,i}_h) 
                &= (\bm f^n,\bm v_h),  \\
 (\nabla\cdot\bm u^{n,i}_h,\phi_h)+\tfrac{1}{\lambda}(\xi^{n,i}_h,\phi_h)   
                &= \tfrac{\alpha}{\lambda}(p^{n,i-1}_h,\phi_h)+\tfrac{\beta}{\lambda}(T^{n,i-1}_h,\phi_h).  \\ 
\end{aligned}
\end{equation}
 Given $T_h^{n,{i-1}}$, $\xi_h^{n,i-1}$, find $p_h^{n,i}$ such that for all $q_h\in W_h$,
\begin{equation}\label{alg: H_F_M 2} 
  \begin{aligned}  
        c_\alpha(p_{h}^{n,i},q_h)     
      + \Delta t(\bm K\nabla p^{n,i}_h,\nabla q_h)+ \rho_p(p^{n,i}_h-p^{n,i-1}_h,  q_h)&  =\\
 \tfrac\alpha\lambda(\xi_{h}^{n,i-1},q_h)- c_{\alpha\beta}(T_{h}^{n,i-1},q_h)&+  (\tilde{g}, q_h)       .\\
\end{aligned}
\end{equation}
 Given $p_h^{n,{i-1}}$, $T_h^{n,{i-1}}$, $\xi_h^{n,i-1}$, find $ T_h^{n,i} $ such that for all $ S_h \in W_h$,
\begin{equation}\label{alg: H_F_M 3} 
  \begin{aligned}  
     c_\beta(T_h^{n,i},S_h) 
             +\Delta t(\bm\Theta\nabla T^{n,i}_h,\nabla S_h) + \rho_T(T^{n,i}_h-T^{n,i-1}_h,  S_h)
               = \tfrac\beta\lambda(\xi_h^{n,i-1},S_h) & \\
           -c_{\alpha\beta}(p_h^{n,i-1},S_h)
            +  (\tilde{H}_s, S_h) 
            +\Delta t(c_f\bm{K}\nabla p_h^{n,i-1}\cdot  \nabla T_h^{n,i-1}, S_h)&.\\
\end{aligned}
\end{equation}

The algorithms M-FH, MF-H, MH-F, and M-F-H proceed as follows:
1. Compute all subproblems in parallel at each iteration step.
2. Repeat the iteration process until the sequence \((\bm{u}_h^{n,i}, \xi_h^{n,i}, p_h^{n,i}, T_h^{n,i})\) converges to within a prescribed tolerance. 
 
Next, we present a convergence theorem for the proposed iterative algorithms. The proof of the theorem is presented in Section \ref{sec: the error estimate and convergence}.
\begin{theorem}\label{th: convergence of all alg} 
Assuming the \textbf{Assumptions} (A1)-(A4) hold and the time step satisfy
$$ \Delta t<\min\{
          \frac{( c_0-b_0 )h^2}{   c_fk_MC_u^2C_{inv}^2 - k_mC_p^2h^2 },
          ~~
          \frac{ (a_0-b_0 )h^2 }{ c_f  k_MC_u^2C_{inv}^2
                                  +\tfrac12 c_f h^2
                                - \theta_mC_p^2h^2}    \},$$
where $C_u$, $C_{inv}$, and $C_p$ are defined in Section \ref{sec: the error estimate and convergence}. Then, the solutions generated by the MFH and the M-FH algorithms both converge to the solutions of equations \eqref{eq:coupled alg}.
The solutions generated by the MF-H algorithm converge to the solutions of equations \eqref{eq:coupled alg}, provided the stabilization parameters $\rho_p\ge |c_{\alpha\beta}|+\tfrac{\alpha}{\lambda}$, $\rho_T\ge|c_{\alpha\beta}|$, and $\rho_\xi\ge\tfrac{\alpha}{\lambda}$. The solutions generated by the MH-F algorithm converge to the solutions of equations \eqref{eq:coupled alg}, provided the stabilization parameters $\rho_p\ge |c_{\alpha\beta}|$, $\rho_T\ge|c_{\alpha\beta}|+\tfrac{\alpha}{\lambda}$, and $\rho_\xi\ge\tfrac{\alpha}{\lambda}$. The solutions generated by the MH-F algorithm converge to the solutions of equations \eqref{eq:coupled alg}, provided the stabilization parameters $\rho_p\ge |c_{\alpha\beta}|$ and $\rho_T\ge|c_{\alpha\beta}|$.

\end{theorem}

\section{Proofs of Error Estimates and Convergence Analysis}\label{sec: the error estimate and convergence}
‌In this section, we present the proofs for the estimates of the nonlinear model \eqref{eq:coupled alg} and the convergence of the proposed iterative algorithms.

We introduce suitable projection operators.  
For the discrete finite element spaces \( \bm{V}_h \), \( Q_h \) satisfying inf-sup condition, and \( W_h \) defined in subsection~\ref{eq:coupled alg}, we define projections
\[
\Pi^{\bm{V}_h}: \bm{V} \to \bm{V}_h, \quad 
\Pi^{Q_h}: Q \to Q_h, \quad 
\Pi^{W_{p,h}}: W \to W_h, \quad 
\Pi^{W_{T,h}}: W \to W_h,
\]
which satisfy the following relations for all 
\((\bm v_h,\phi_h,p_h,T_h)\in \bm V_h\times Q_h\times W_h\times W_h\):
\begin{equation}\label{Pi V}
   2\mu(\bm\varepsilon(\Pi^{\bm V_h}\bm u ),\bm\varepsilon( \bm v_h ))
   - (\nabla\cdot\bm v_h,\Pi^{Q_h} \xi)
   = 2\mu(\bm\varepsilon( \bm u ),\bm\varepsilon( \bm v_h ))
   - (\nabla\cdot\bm v_h, \xi),
\end{equation} 
\begin{equation}\label{Pi Q}
   (\nabla\cdot\Pi^{\bm V_h}\bm u,\phi_h)+\tfrac1\lambda(\Pi^{Q_h} \xi,\phi_h) = (\nabla\cdot \bm u,\phi_h)+\tfrac1\lambda( \xi,\phi_h),
\end{equation} 
\begin{equation}\label{Pi Wp}
   (\bm K\nabla\Pi^{W_{p,h}}p,\nabla q_h) = (\bm K\nabla p,\nabla q_h),
\end{equation} 
\begin{equation}\label{Pi WT}
   (\bm \Theta\nabla\Pi^{W_{T,h}}T,\nabla S_h) = (\bm\Theta\nabla T,\nabla S_h).
\end{equation}
Furthermore, for the operators $\Pi^{\bm V_h}$,~$\Pi^{Q_h} $,~$\Pi^{W_{p,h}}$, and $\Pi^{W_{T,h}}$, if $\bm u\in \bm H_{0}^{k+1}(\Omega)$, $\xi\in H^{s+1}(\Omega)$, $p,~T\in H_0^{l+1}$ we have following
 properties. 
\begin{equation}\label{Pi V Q properties}
   \|\nabla(\Pi^{\bm V_h}\bm u-\bm u)\|_{L^2(\Omega)}+\|\Pi^{Q_h} \xi-\xi\|_{L^2(\Omega)}
   \lesssim h^k\|\bm u\|_{H^{k+1}(\Omega)}+h^{s+1}\|\xi\|_{H^{s+1}(\Omega)}, 
\end{equation} 
\begin{equation}\label{Pi Wp properties}
   \| \Pi^{W_{p,h}}p-p \|_{L^2(\Omega)}+h\|\nabla(\Pi^{W_{p,h}}p-p)\|_{L^2(\Omega)}
   \lesssim h^{l+1} \|p\|_{H^{l+1}(\Omega)},
\end{equation} 
and
\begin{equation}\label{Pi WT properties}
     \| \Pi^{W_{T,h}}T-T \|_{L^2(\Omega)}+h\|\nabla(\Pi^{W_{T,h}}T-T)\|_{L^2(\Omega)}
   \lesssim h^{l+1} \|T\|_{H^{l+1}(\Omega)}.
\end{equation} 

We also need to cite the following two lemmas, which are frequently used in our subsequent analysis. These two lemmas are easy to obtain, and their proofs can also be found in \cite{gu2023priori}.
\begin{lemma}\label{le: B(u,v)}
    Let \(B\) be a symmetric and positive semi-definite bilinear form. Then \[     
    2B(u,u-v)=B(u,u)-B(v,v)+B(u-v,u-v),     \] is satisfied, which immediately implies the following inequality \[
    2B(u,u-v)\ge B(u,u)-B(v,v).
    \]
\end{lemma}
 
For the subsequent analysis, throughout this work, we adopt the following notation:
Let $e_{\bm u}^n$, $e_{\xi}^n$, $e_{p}^n$, $e_{T}^n$ represent the discrepancies between the exact solutions of \eqref{TP_Model_var after parameters transformation} at temporal index $n$ and the numerical approximations computed via \eqref{TP_Model_dis_a}-\eqref{TP_Model_dis_d}. 
Each error component is partitioned into interpolation and discretization components as follows:
\begin{equation}\label{eq: e=eI+eh}
\begin{aligned} 
 &e_{\bm u}^n:=\bm u^n-\bm u_h^n=( \bm u^n-\Pi^{\bm V_h}\bm u^n)+(\Pi^{\bm V_h}\bm u^n -\bm u_h^n):=e_{\bm u}^{I,n}+e_{\bm u}^{h,n},\\
 &e_{\xi}^n:=\xi^n-\xi_h^n=(\xi^n-\Pi^{Q_h}\xi^n)+(\Pi^{Q_h}\xi^n -\xi_h^n):=e_{\xi}^{I,n}+e_{\xi}^{h,n},\\ 
 &e_{p}^n:=p^n-p_h^n=(p^n-\Pi^{W_{p,h}} p^n)+(\Pi^{W_{p,h}} p^n -p_h^n):=e_{p}^{I,n}+e_{p}^{h,n},\\ 
 &e_{T}^n:=T^n-T_h^n=(T^n-\Pi^{W_{T,h}} T^n)+(\Pi^{W_{T,h}} T^n -T_h^n):=e_{T}^{I,n}+e_{T}^{h,n}.\\
\end{aligned}
\end{equation} 
Specifically, we set $\bm u_h^0=\Pi^{\bm V_h}\bm u^0$, $\xi_h^0=\Pi^{Q_h}\xi^0$, 
$p_h^0=\Pi^{W_{p,h}}p^0$, and $T_h^0=\Pi^{W_{T,h}}T^0$.
Moreover, we define the temporal increments as:
\begin{equation}\label{diff_real_num D}
\begin{aligned} 
     (d_{\bm u}^{n+1},~d_{\xi}^{n+1},~d_{p}^{n+1},~d_{T}^{n+1}):=
     (e_{\bm u}^{h,n+1},~e_{\xi}^{h,n+1},~e_{p}^{h,n+1},~e_{T}^{h,n+1})
    -(e_{\bm u}^{h,n},~e_{\xi}^{h,n},~e_{p}^{h,n},~e_{T}^{h,n}).
\end{aligned}
\end{equation}

We now state the error estimate of the algorithm in equations \eqref {TP_Model_dis_a}-\eqref{TP_Model_dis_d}.

\begin{proof}[\textbf{The proof of Theorem \ref{th: error estimate}:}]
Letting the test functions be the discrete test functions and subtracting \eqref{TP_Model_var after parameters transformation} from \eqref{TP_Model_dis_a}-\eqref{TP_Model_dis_d} at \(t=t_{n+1}\), we obtain 
\begin{equation} 
\begin{aligned} 
2\mu(\bm\varepsilon(e_{\bm u}^{n+1}), \bm\varepsilon(\bm v_h)) 
                        - (\nabla\cdot\bm v_h,e_{\xi}^{n+1})&=0,  \\
  (\nabla\cdot e_{\bm u}^{n+1} ,\phi_h)
  +\tfrac{1}{\lambda}(e_{\xi}^{n+1} ,\phi_h)
  -\tfrac{\alpha}{\lambda}(e_{p}^{n+1} ,\phi_h)
  -\tfrac{\beta}{\lambda}(e_{T}^{n+1} ,\phi_h)&=0 ,\\
 -\tfrac\alpha\lambda( \xi_t^{n+1} - \tfrac{\xi^{n+1}_h-\xi^{n}_h}{\Delta t} ,q_h)
 +c_\alpha( p_t^{n+1}-\tfrac{p^{n+1}_h-p^{n}_h}{\Delta t} ,q_h)
 + (\bm K\nabla e_p^{n+1},\nabla q_h)&\\
 + c_{\alpha\beta}( T_t^{n+1}-\tfrac{T^{n+1}_h-T^{n}_h}{\Delta t} ,q_h)   &=0,\\
 -\tfrac\beta\lambda( \xi^{n+1}_t- \tfrac{\xi^{n+1}_h-\xi^{n}_h}{\Delta t},S_h)
 +c_\beta( T_t^{n+1}-\tfrac{T^{n+1}_h-T^{n}_h}{\Delta t} ,S_h) 
  +(\bm\Theta\nabla e_T^{n+1},\nabla S_h)&\\
-c_f( \bm K \nabla p^{n+1}\cdot\nabla T^{n+1}-\bm K \nabla p_h^{n+1}\cdot\nabla T_h^{n+1} ,S_h)
  +c_{\alpha\beta}( p_t^{n+1}-\tfrac{p^{n+1}_h-p^{n}_h}{\Delta t} ,S_h)&=0.\\
\end{aligned}
\end{equation}  
Using the projection operators introduced in \eqref{Pi V}, \eqref{Pi Q}, \eqref{Pi Wp}, and \eqref{Pi WT}, the above equations can be rewritten as
\begin{equation}\label{eq: evun+1-evxin+1=0}
\begin{aligned} 
2\mu(\bm\varepsilon(e_{\bm u}^{h,n+1}), \bm\varepsilon(\bm v_h)) 
                        - (\nabla\cdot\bm v_h,e_{\xi}^{h,n+1})&=0,  \\
\end{aligned}
\end{equation}  
\begin{equation}\label{eq: -eun+1-exin+1+pn+1+Tn+1=0}
\begin{aligned} 
    (\nabla\cdot e_{\bm u}^{h,n+1} ,\phi_h)
  +\tfrac{1}{\lambda}(e_{\xi}^{h,n+1} ,\phi_h)
  -\tfrac{\alpha}{\lambda}(e_{p}^{n+1},\phi_h)
  -\tfrac{\beta}{\lambda}(e_{T}^{n+1} ,\phi_h)&=0 ,\\
\end{aligned}
\end{equation} 
\begin{equation}\label{eq: reaction dis Dt a}
\begin{aligned} 
 &c_\alpha( d_{p}^{n+1},q_h)
 + \Delta t(\bm K\nabla e_p^{h,n+1},\nabla q_h)
 + c_{\alpha\beta}(d_{T}^{n+1},q_h) -\tfrac{\alpha}{\lambda}(d_\xi^{n+1},q_h) \\  
 =&\tfrac\alpha\lambda( \Delta t\xi_t^{n+1} -  \Pi^{Q_h}\xi^{n+1} + \Pi^{Q_h}\xi^{n},q_h) \\
 & +c_\alpha( \Pi^{W_{p,h}} p^{n+1}-\Pi^{W_{p,h}} p^{n}-\Delta tp_t^{n+1},q_h) \\
 & + c_{\alpha\beta}(\Pi^{W_{T,h}} T^{n+1}-\Pi^{W_{T,h}} T^{n}-\Delta tT_t^{n+1},q_h),\\
 \end{aligned}
\end{equation}  
\begin{equation}\label{eq: reaction dis Dt b}
\begin{aligned} 
& c_\beta(d_{T}^{n+1},S_h)   
  +\Delta t(\bm\Theta\nabla e_T^{h,n+1},\nabla S_h)
  +c_{\alpha\beta}( d_{p}^{n+1},S_h)  -\tfrac\beta\lambda(d_\xi^{n+1},S_h) \\
 =& \tfrac\beta\lambda( \Delta t\xi_t^{n+1}- \Pi^{Q_h}\xi^{n+1} +\Pi^{Q_h}\xi^{n} ),S_h) \\
 &  +c_{\alpha\beta}(\Pi^{W_{p,h}} p^{n+1}-\Pi^{W_{p,h}} p^{n}-\Delta tp_t^{n+1} ,S_h) \\
&   +c_\beta(\Pi^{W_{T,h}} T^{n+1}-\Pi^{W_{T,h}} T^{n}-\Delta tT_t^{n+1},S_h)\\
&+\Delta tc_f(\bm K \nabla p^{n+1}\cdot\nabla T^{n+1}- \bm K \nabla p_h^{n+1}\cdot\nabla T_h^{n+1},S_h) .\\
\end{aligned}
\end{equation}  
Differentiating the second equation of \ref{TP_Model_var after parameters transformation} with respect to $t$ at the
$(n+1)$-th time step, then multiplying the resulted equation by $\Delta t$, we derive that
\begin{equation} \label{eq:Dt  -u-xi+p+T=0}
\begin{aligned}  
  -(\nabla\cdot (\Delta t\bm u_t^{n+1}),\phi)-\tfrac{1}{\lambda}(\Delta t\xi_t^{n+1},\phi)+\tfrac{\alpha}{\lambda}(\Delta tp_t^{n+1},\phi)+\tfrac{\beta}{\lambda}(\Delta tT_t^{n+1},\phi)&=0. 
\end{aligned}
\end{equation}  
For \eqref{eq: -eun+1-exin+1+pn+1+Tn+1=0} and the second equation of \eqref{TP_Model_var after parameters transformation},  taking a difference between $t=t_{n+1}$ and $t=t_n$, we have
\begin{equation}\label{eq:Dt -eun+1-exin+1+pn+1+Tn+1=0}
    -(\nabla\cdot d_{\bm u}^{n+1},\phi_h)
  -\tfrac{1}{\lambda}(d_{\xi}^{n+1},\phi_h)
  +\tfrac{\alpha}{\lambda}(e_{p}^{n+1}-e_{p}^{n},\phi_h)
  +\tfrac{\beta}{\lambda}(e_{T}^{n+1}-e_{T}^{n} ,\phi_h) =0 , 
\end{equation}  
and 
\begin{equation}\label{eq:n+1n -u-xi+p+T=0} 
  -(\nabla\cdot(\bm u^{n+1}-\bm u^n),\phi)-\tfrac{1}{\lambda}(\xi^{n+1}-\xi^n,\phi)
+\tfrac{\alpha}{\lambda}(p^{n+1}-p^n,\phi)+\tfrac{\beta}{\lambda}(T^{n+1}-T^n,\phi) =0. 
\end{equation}  
By using the definitions of $d_{p}^{n+1}$ and $d_{T}^{n+1}$ we reformulate \eqref{eq:Dt -eun+1-exin+1+pn+1+Tn+1=0} as
\begin{equation} \label{eq:reform Dt -eun+1-exin+1+pn+1+Tn+1=0} 
\begin{aligned}  
  &  (\nabla\cdot d_{\bm u}^{n+1},\phi_h)
  +\tfrac{1}{\lambda}(d_{\xi}^{n+1},\phi_h)
  -\tfrac{\alpha}{\lambda}( d_{p}^{n+1},\phi_h)
  -\tfrac{\beta}{\lambda}(  d_{T}^{n+1},\phi_h) \\
= & \tfrac{\alpha}{\lambda}(p^{n+1}-p^n  ,\phi_h)
 -\tfrac{\alpha}{\lambda}( \Pi^{W_{p,h}} p^{n+1}-\Pi^{W_{p,h}} p^{n} ,\phi_h) \\
&+\tfrac{\beta}{\lambda}( T^{n+1}-T^n,\phi_h)-\tfrac{\beta}{\lambda}(\Pi^{W_{T,h}} T^{n+1}-\Pi^{W_{T,h}} T^{n},\phi_h).
\end{aligned}
\end{equation}  
Taking $\phi=\phi_h$ in \eqref{eq:Dt  -u-xi+p+T=0} and \eqref{eq:n+1n -u-xi+p+T=0}, we combine \eqref{eq:reform Dt -eun+1-exin+1+pn+1+Tn+1=0} to get
\begin{equation} \label{eq:reform Dt -eun+1-exin+1+pn+1+Tn+1=0 +u} 
\begin{aligned}  
  &   (\nabla\cdot d_{\bm u}^{n+1},\phi_h)
  +\tfrac{1}{\lambda}(d_{\xi}^{n+1},\phi_h)
  -\tfrac{\alpha}{\lambda}( d_{p}^{n+1},\phi_h)
  -\tfrac{\beta}{\lambda}(  d_{T}^{n+1},\phi_h) \\
=& (\nabla\cdot(\bm u^{n+1}-\bm u^n-\Delta t\bm u_t^{n+1}),\phi_h)\\
&  +  \tfrac{1}{\lambda}( \xi^{n+1}- \xi^{n}-\Delta t\xi_t^{n+1}, \phi_h)  \\
&   -\tfrac{\alpha}{\lambda}( \Pi^{W_{p,h}} p^{n+1}-\Pi^{W_{p,h}} p^{n}-\Delta tp_t^{n+1} ,\phi_h) \\
& -\tfrac{\beta}{\lambda}(\Pi^{W_{T,h}} T^{n+1}-\Pi^{W_{T,h}} T^{n}-\Delta tT_t^{n+1},\phi_h),
\end{aligned}
\end{equation}  
Choosing $\bm v_h=d_{\bm u}^{n+1}$ in \eqref{eq: evun+1-evxin+1=0}, $\phi_h=e_\xi^{h,n+1}$ in \eqref{eq:reform Dt -eun+1-exin+1+pn+1+Tn+1=0 +u}, $q_h=e_p^{h,n+1}$ in \eqref{eq: reaction dis Dt a} and  $S_h=e_T^{h,n+1}$ in \eqref{eq: reaction dis Dt b}, we derive 
\begin{equation}\label{eq: evun+1-evxin+1=0 vh=dun+1}
\begin{aligned} 
2\mu(\bm\varepsilon(e_{\bm u}^{h,n+1}), \bm\varepsilon(d_{\bm u}^{n+1})) 
                        - (\nabla\cdot d_{\bm u}^{n+1},e_{\xi}^{h,n+1})&=0,  \\
\end{aligned}
\end{equation}  
\begin{equation} \label{eq:reform Dt -eun+1-exin+1+pn+1+Tn+1=0 phih=exi} 
\begin{aligned}  
  &  (\nabla\cdot d_{\bm u}^{n+1},e_\xi^{h,n+1})
  +\tfrac{1}{\lambda}(d_{\xi}^{n+1},e_\xi^{h,n+1})
  -\tfrac{\alpha}{\lambda}( d_{p}^{n+1},e_\xi^{h,n+1})
  -\tfrac{\beta}{\lambda}(  d_{T}^{n+1},e_\xi^{h,n+1}) \\
=&  (\nabla\cdot(\bm u^{n+1}-\bm u^n-\Delta t\bm u_t^{n+1}),e_\xi^{h,n+1}) \\
 &  +  \tfrac{1}{\lambda}(\xi^{n+1}-\xi^{n}-\Delta t\xi_t^{n+1},e_\xi^{h,n+1})  \\
&   -\tfrac{\alpha}{\lambda}( \Pi^{W_{p,h}} p^{n+1}-\Pi^{W_{p,h}} p^{n}-\Delta tp_t^{n+1} ,e_\xi^{h,n+1}) \\
& -\tfrac{\beta}{\lambda}(\Pi^{W_{T,h}} T^{n+1}-\Pi^{W_{T,h}} T^{n}-\Delta tT_t^{n+1},e_\xi^{h,n+1}),
\end{aligned}
\end{equation}  
 \begin{equation}\label{eq: reaction dis Dt a qh=ep}
\begin{aligned} 
 &c_\alpha( d_{p}^{n+1},e_p^{h,n+1})
 + \Delta t(\bm K\nabla e_p^{h,n+1},\nabla e_p^{h,n+1})
 + c_{\alpha\beta}(d_{T}^{n+1},e_p^{h,n+1})  -\tfrac{\alpha}{\lambda}(d_\xi^{n+1},e_p^{h,n+1})\\  
= &\tfrac\alpha\lambda( \Delta t\xi_t^{n+1} - \Pi^{Q_h}\xi^{n+1}+\Pi^{Q_h}\xi^{n} ,e_p^{h,n+1}) \\
&  +c_\alpha( \Pi^{W_{p,h}} p^{n+1}-\Pi^{W_{p,h}} p^{n}-\Delta tp_t^{n+1},e_p^{h,n+1}) \\
 & + c_{\alpha\beta}(\Pi^{W_{T,h}} T^{n+1}-\Pi^{W_{T,h}} T^{n}-\Delta tT_t^{n+1},e_p^{h,n+1}),\\
 \end{aligned}
\end{equation}  
and
\begin{equation}\label{eq: reaction dis Dt b Sh=eT}
\begin{aligned} 
& c_\beta(d_{T}^{n+1},e_T^{h,n+1})   
  +\Delta t(\bm\Theta\nabla e_T^{h,n+1},\nabla e_T^{h,n+1})
  +c_{\alpha\beta}( d_{p}^{n+1},e_T^{h,n+1})   -\tfrac\beta\lambda(d_\xi^{n+1},e_T^{h,n+1})\\
 =& \tfrac\beta\lambda( \Delta t\xi_t^{n+1}- \Pi^{Q_h}\xi^{n+1}+\Pi^{Q_h}\xi^{n} ,e_T^{h,n+1}) \\
 &  +c_{\alpha\beta}(\Pi^{W_{p,h}} p^{n+1}-\Pi^{W_{p,h}} p^{n}-\Delta tp_t^{n+1} ,e_T^{h,n+1}) \\
&   +c_\beta(\Pi^{W_{T,h}} T^{n+1}-\Pi^{W_{T,h}} T^{n}-\Delta tT_t^{n+1},e_T^{h,n+1})\\
&+\Delta tc_f(\bm K \nabla p^{n+1} \cdot\nabla T^{n+1}-\bm K \nabla p^{n+1}_h \cdot\nabla T^{n+1}_h,e_T^{h,n+1}) .\\
\end{aligned}
\end{equation}  
Summing equations \eqref{eq: evun+1-evxin+1=0 vh=dun+1}, \eqref{eq:reform Dt -eun+1-exin+1+pn+1+Tn+1=0 phih=exi}, \eqref{eq: reaction dis Dt a qh=ep}, and \eqref{eq: reaction dis Dt b Sh=eT} over the index $n$  from  $0$  to  $N-1$ yields:
\begin{equation}\label{eq: LHS=Ei}
\begin{aligned} 
LHS
:=&\sum_{n=0}^{N-1}\Big(
2\mu(\bm\varepsilon(e_{\bm u}^{h,n+1}), \bm\varepsilon(d_{\bm u}^{n+1}))   
  +\tfrac{1}{\lambda}(d_{\xi}^{n+1},e_\xi^{h,n+1}) 
  -\tfrac{\alpha}{\lambda}( d_{p}^{n+1},e_\xi^{h,n+1})\\
 & -\tfrac{\beta}{\lambda}(  d_{T}^{n+1},e_\xi^{h,n+1}) 
  +c_\alpha( d_{p}^{n+1},e_p^{h,n+1})
 + \Delta t(\bm K\nabla e_p^{h,n+1},\nabla e_p^{h,n+1})
    \\
& + c_{\alpha\beta}(d_{T}^{n+1},e_p^{h,n+1})+ c_\beta(d_{T}^{n+1},e_T^{h,n+1})   
  +\Delta t(\bm\Theta\nabla e_T^{h,n+1},\nabla e_T^{h,n+1})
  \\ 
&+c_{\alpha\beta}( d_{p}^{n+1},e_T^{h,n+1})  
 -\tfrac{\alpha}{\lambda}(d_\xi^{n+1},e_p^{h,n+1})
 -\tfrac\beta\lambda(d_\xi^{n+1},e_T^{h,n+1})\Big)\\
=&\sum_{i=1}^{11}E_i,
\end{aligned}
\end{equation}  
where 
$$
E_1= (\nabla\cdot(\bm u^{n+1}-\bm u^n-\Delta t\bm u_t^{n+1}),e_\xi^{h,n+1})  ,
$$
$$
E_2=    \tfrac{1}{\lambda}( \xi^{n+1}-\xi^{n}-\Delta t\xi_t^{n+1},e_\xi^{h,n+1}) ,
$$
$$
E_3=  -\tfrac{\alpha}{\lambda}( \Pi^{W_{p,h}} p^{n+1}-\Pi^{W_{p,h}} p^{n}-\Delta tp_t^{n+1} ,e_\xi^{h,n+1}),
$$
$$
E_4=   -\tfrac{\beta}{\lambda}(\Pi^{W_{T,h}} T^{n+1}-\Pi^{W_{T,h}} T^{n}-\Delta tT_t^{n+1},e_\xi^{h,n+1}),
$$
$$
E_5=\tfrac\alpha\lambda( \Delta t\xi_t^{n+1} - \Pi^{Q_h}\xi^{n+1} +\Pi^{Q_h}\xi^{n}  ,e_p^{h,n+1})  ,
$$
$$
E_6= c_\alpha( \Pi^{W_{p,h}} p^{n+1}-\Pi^{W_{p,h}} p^{n}-\Delta tp_t^{n+1},e_p^{h,n+1}) ,
$$
$$
E_7=  c_{\alpha\beta}(\Pi^{W_{T,h}} T^{n+1}-\Pi^{W_{T,h}} T^{n}-\Delta tT_t^{n+1},e_p^{h,n+1}),
$$
$$
E_{8}= \tfrac\beta\lambda( \Delta t\xi_t^{n+1}- \Pi^{Q_h}\xi^{n+1} +\Pi^{Q_h}\xi^{n} ,e_T^{h,n+1}) ,
$$
$$
E_{9}= c_{\alpha\beta}(\Pi^{W_{p,h}} p^{n+1}-\Pi^{W_{p,h}} p^{n}-\Delta tp_t^{n+1} ,e_T^{h,n+1}),
$$
$$
E_{10}=  c_\beta(\Pi^{W_{T,h}} T^{n+1}-\Pi^{W_{T,h}} T^{n}-\Delta tT_t^{n+1},e_T^{h,n+1}),
$$
and
$$
E_{11}= \Delta tc_f(\bm K \nabla p^{n+1} \cdot\nabla T^{n+1}-\bm K \nabla p^{n+1}_h \cdot\nabla T^{n+1}_h,e_T^{h,n+1}).
$$
Rearranging \(LHS\) and applying Lemma \ref{le: B(u,v)}, we can estimate the lower bound of $LHS$ as following:
\begin{equation*} 
\begin{aligned}  
&\sum_{n=0}^{N-1}\Big( 
      \mu \|\bm\varepsilon(e_{\bm u}^{h,n+1} )\|^2_{L^2(\Omega)}
              -\mu \|\bm\varepsilon(e_{\bm u}^{h,n} )\|^2_{L^2(\Omega)} \\
  &+\tfrac{1}{2\lambda} \|\ e_{\xi}^{h,n+1}-\alpha e_p^{h,n+1}-\beta 
                      e_T^{h,n+1}\|^2_{L^2(\Omega)}    
   -\tfrac{1}{2\lambda} \|\ e_{\xi}^{h,n}-\alpha e_p^{h,n}-\beta 
                      e_T^{h,n}\|^2_{L^2(\Omega)}  \\
  & +\tfrac{c_0-b_0}{2} \| e_p^{h,n+1} \|^2_{L^2(\Omega)}
                -\tfrac{c_0-b_0}{2}  \| e_p^{h,n} \|^2_{L^2(\Omega)}\\
  & +\tfrac{a_0-b_0}{2} \| e_T^{h,n+1} \|^2_{L^2(\Omega)} 
            -\tfrac{a_0-b_0}{2}  \| e_T^{h,n} \|^2_{L^2(\Omega)} 
    \\
  & + \tfrac{b_0}{2} \| e_p^{h,n+1}-e_T^{h,n+1} \|^2_{L^2(\Omega)} 
           - \tfrac{b_0}{2} \| e_p^{h,n}-e_T^{h,n} \|^2_{L^2(\Omega)} 
           \Big)  \\ 
  &  +\Delta t\sum_{n=0}^{N-1}\Big( 
         k_m \|\nabla e_p^{h,n+1} \|^2_{L^2(\Omega)} + \theta_m \|\nabla e_T^{h,n+1} \|^2_{L^2(\Omega)}
        \Big) \\
\le &\sum_{n=0}^{N-1}\Big( 2\mu(\bm\varepsilon(e_{\bm u}^{h,n+1}), 
                       \bm\varepsilon(d_{\bm u}^{n+1}))\\
  &+\tfrac1\lambda(d_{\xi}^{n+1}-\alpha d_{p}^{n+1}-\beta d_{T}^{n+1}
                ,\ e_{\xi}^{h,n+1}-\alpha e_p^{h,n+1}-\beta e_T^{h,n+1})
                \\ 
  & + (c_0-b_0) ( d_{p}^{n+1} ,e_p^{h,n+1})
    + (a_0-b_0) ( d_{T}^{n+1} ,e_T^{h,n+1})   
    \\
  &  +  b_0( d_{p}^{n+1}-d_{T}^{n+1} ,e_p^{h,n+1}-e_T^{h,n+1})\\
  &  + \Delta t(\bm K\nabla e_p^{h,n+1},\nabla e_p^{h,n+1})  
    +\Delta t(\bm\Theta\nabla e_T^{h,n+1},\nabla e_T^{h,n+1}) 
       \Big) \\
=& LHS  ,\\ 
\end{aligned}
\end{equation*}  
which directly implies that
\begin{equation}\label{ineq: LHS ge etc.}
\begin{aligned}  
 & \tfrac12\Big( 
      2\mu \|\bm \varepsilon(e_{\bm u}^{h,N}) \|^2_{L^2(\Omega)}
                    -\mu \|\bm\varepsilon(e_{\bm u}^{h,0}  )\|^2_{L^2(\Omega)} \\
  &+\tfrac{1}{\lambda} \|\ e_{\xi}^{h,N}-\alpha e_p^{h,N}-\beta 
                      e_T^{h,N}\|^2_{L^2(\Omega)}    
   -\tfrac{1}{\lambda} \|\ e_{\xi}^{h,0}-\alpha e_p^{h,0}-\beta 
                      e_T^{h,0}\|^2_{L^2(\Omega)}  \\
  & +(c_0-b_0) \| e_p^{h,N} \|^2_{L^2(\Omega)}
                -(c_0-b_0)   \| e_p^{h,0} \|^2_{L^2(\Omega)}\\
  & +(a_0-b_0) \| e_T^{h,N} \|^2_{L^2(\Omega)} 
            -(a_0-b_0)  \| e_T^{h,0} \|^2_{L^2(\Omega)} 
    \\
  & +  b_0  \| e_p^{h,N}-e_T^{h,N} \|^2_{L^2(\Omega)} 
           -  b_0  \| e_p^{h,0}-e_T^{h,0} \|^2_{L^2(\Omega)} 
           \Big)  \\ 
  &  +\Delta t\sum_{n=0}^{N-1}\Big( 
           k_m \|\nabla e_p^{h,n+1} \|^2_{L^2(\Omega)} + \theta_m \|\nabla e_T^{h,n+1} \|^2_{L^2(\Omega)}
        \Big)  \le LHS.\\ 
\end{aligned}
\end{equation}

Next, by applying Taylor series expansion‌,  the Cauchy-Schwarz inequality,
 and Young’s inequality, we can bound $E_1$ as:
\begin{equation*} 
\begin{aligned}  
E_1&=\sum_{n=0}^{N-1}(\nabla\cdot(\bm u^{n+1}-\bm u^{n}-\Delta t\bm 
               u_t^{n+1})),e_{\xi}^{h,n+1})\\ 
&\le \tfrac{1}{\Delta t}  \tfrac{1}{2}
\sum_{n=0}^{N-1} \| \bm u^{n+1} -\bm u^{n} - \Delta t 
      \bm u_t^{n+1} \|_{H^1(\Omega)}^2 
+ \tfrac{1}{2}\Delta t \sum_{n=0}^{N-1} \| e_{\xi}^{h,n+1} \|_{L^2(\Omega)}^2 \\
&= \tfrac{1}{\Delta t}  \tfrac{1}{2}\sum_{n=0}^{N-1} \| \int_{t_{n}}^{t_{n+1}}\bm u_{tt}(s) (s-t_{n}) ds \|_{H^1(\Omega)}^2 
+ \tfrac{1}{2}\Delta t  \sum_{n=0}^{N-1} \| e_{\xi}^{h,n+1} \|_{L^2(\Omega)}^2 \\
&\le \tfrac{1}{2}\Delta t^2  \int_{0}^{\tau} \left\| \bm u_{tt} \right\|_{H^1(\Omega)}^2 ds 
+ \tfrac{1}{2}\Delta t \sum_{n=0}^{N-1} \| e_{\xi}^{h,n+1} \|_{L^2(\Omega)}^2,
\end{aligned}
\end{equation*} 
and bound $E_2$ as:
 \begin{equation*} 
\begin{aligned}  
E_2=&\sum_{n=0}^{N-1} \tfrac{1}{\lambda}( \xi^{n+1}- \xi^{n}-\Delta t\xi_t^{n+1},e_{\xi}^{h,n+1})\\ 
  \le&  \tfrac{1}{2}\Delta t^2  \int_{0}^{\tau} \left\| \xi_{tt} \right\|_{L^2(\Omega)}^2 ds 
+ \tfrac{1}{2}\Delta t \sum_{n=0}^{N-1} \| e_{\xi}^{h,n+1} \|_{L^2(\Omega)}^2,                      
\end{aligned}
\end{equation*}

By applying Taylor series expansion‌, the Cauchy-Schwarz inequality, Young’s inequality, and the properties of the operators in \eqref{Pi V Q properties}-\eqref{Pi WT properties}, we can bound $E_3$ as:
 \begin{equation*} 
\begin{aligned}  
E_3=&\sum_{n=0}^{N-1} -\tfrac{\alpha}{\lambda}( \Pi^{W_{p,h}} p^{n+1}-\Pi^{W_{p,h}} p^{n}-\Delta tp_t^{n+1} ,e_\xi^{h,n+1})\\
=&\sum_{n=0}^{N-1}-\tfrac{\alpha}{\lambda}((\Pi^{W_{p,h}} p^{n+1}-\Pi^{W_{p,h}} p^{n})
       -(p^{n+1}-p^{n})+(p^{n+1}-p^{n}-\Delta t p_t^{n}), e_{\xi}^{h,n+1})   \\
\le& \sum_{n=0}^{N-1} \tfrac{\alpha}{\lambda}\Big(   
  \tfrac{1}{\Delta t}h^{2l+2}\|p^{n+1}-p^{n} \|_{L^2(\Omega)}^2
+\tfrac{1}{\Delta t}\|p^{n+1}-p^{n}-\Delta t p_t^{n}\|_{L^2(\Omega)}^2\\
&\quad\quad+\tfrac{1}{2}\Delta t   \| e_{\xi}^{h,n+1} \|_{L^2(\Omega)}^2\Big)\\
\le&
\tfrac\alpha\lambda\Big( h^{2l+2} 
    \int_{0}^{\tau}\| p_t  \|_{H^{l+1}(\Omega)}^2
   +  \Delta t^2   \int_{0}^{\tau}\| p_{tt}  \|_{L^2(\Omega)}^2
    + \tfrac{1}{2}\Delta t \sum_{n=0}^{N-1}        \| e_{\xi}^{h,n+1} \|_{L^2(\Omega)}^2
   \Big), 
\end{aligned}
\end{equation*} 
Similarly, \(E_4-E_{10}\) can be bounded as follows.
\begin{equation*} 
\begin{aligned}  
E_4\le&
\tfrac\beta\lambda\Big( h^{2l+2}
    \int_{0}^{\tau}\| T_t  \|_{H^{l+1}(\Omega)}^2
   +  \Delta t^2   \int_{0}^{\tau}\| T_{tt}  \|_{L^2(\Omega)}^2
    + \tfrac{1}{2}\Delta t  \sum_{n=0}^{N-1}        \| e_{\xi}^{h,n+1} \|_{L^2(\Omega)}^2 
   \Big), \\
\end{aligned}
\end{equation*} 
\begin{equation*} 
\begin{aligned}   
 E_5\le&
\tfrac\alpha\lambda\Big(h^{2k}\| \bm u_t  \|_{H^{k+1}(\Omega)}^2+
     h^{2s+2} \int_{0}^{\tau}\| \xi_t  \|_{H^{s+1}(\Omega)}^2
   +  \Delta t^2  \int_{0}^{\tau}\| \xi_{tt}  \|_{L^2(\Omega)}^2
    + \tfrac{1}{2}\Delta t \sum_{n=0}^{N-1} \| e_{p}^{h,n+1} \|_{L^2(\Omega)}^2
\Big),\\
\end{aligned}
\end{equation*} 
 \begin{equation*} 
\begin{aligned}  
E_6
 \le&
  c_\alpha\left(
       h^{2l+2}\int_{0}^{\tau}\| p_t  \|_{H^{l+1}(\Omega)}^2
   +   \Delta t^2   \int_{0}^{\tau}\| p_{tt}  \|_{L^2(\Omega)}^2
    + \tfrac{1}{2}\Delta t \sum_{n=0}^{N-1}        \| e_{p}^{h,n+1} \|_{L^2(\Omega)}^2
    \right),
\end{aligned}
\end{equation*} 
 \begin{equation*} 
\begin{aligned}   
 E_7
  \le&
  \left|c_{\alpha\beta} \right|\left(   h^{2l+2}
    \int_{0}^{\tau}\| T_t  \|_{H^{l+1}(\Omega)}^2
   +  \Delta t^2  \int_{0}^{\tau}\| T_{tt}  \|_{L^2(\Omega)}^2
    + \tfrac{1}{2}\Delta t \sum_{n=0}^{N-1}        \| e_{p}^{h,n+1} \|_{L^2(\Omega)}^2
    \right),\\
\end{aligned}
\end{equation*} 
\begin{equation*} 
\begin{aligned}   
E_8\le&
\tfrac\beta\lambda\Big(
       h^{2k} \int_{0}^{\tau}\| \bm u_t  \|_{H^{k+1}(\Omega)}^2
    + h^{2s+2} \int_{0}^{\tau}\| \xi_t  \|_{H^{s+1}(\Omega)}^2
   +\Delta t^2   \int_{0}^{\tau}\| \xi_{tt}  \|_{L^2(\Omega)}^2
    + \tfrac{1}{2}\Delta t \sum_{n=0}^{N-1} \| e_{T}^{h,n+1} \|_{L^2(\Omega)}^2
\Big),\\
\end{aligned}
\end{equation*} 
 \begin{equation*} 
\begin{aligned}   
E_9  \le&\left|c_{\alpha\beta} \right|\left(
    h^{2l+2} \int_{0}^{\tau}\| p_t  \|_{H^{l+1}(\Omega)}^2
   + \Delta t^2   \int_{0}^{\tau}\| p_{tt}  \|_{L^2(\Omega)}^2
    + \tfrac{1}{2}\Delta t  \sum_{n=0}^{N-1}   \| e_{T}^{h,n+1} \|_{L^2(\Omega)}^2
    \right),\\
\end{aligned}
\end{equation*} 
and
 \begin{equation*} 
\begin{aligned}   
E_{10} \le&c_\beta\left( 
      h^{2l+2} \int_{0}^{\tau}\| T_t  \|_{H^{l+1}(\Omega)}^2
   +   \Delta t^2   \int_{0}^{\tau}\| T_{tt}  \|_{L^2(\Omega)}^2
    + \tfrac{1}{2}\Delta t \sum_{n=0}^{N-1}        \| e_{T}^{h,n+1} \|_{L^2(\Omega)}^2
    \right).\\
\end{aligned}
\end{equation*} 
In addition to Assumption (A4), we further assume that$ \| \nabla T^{n}  \|_{L^\infty(\Omega)} <C_u$ and
$ \| \nabla p^{n}_h  \|_{L^\infty(\Omega)} <C_u$ where $C_u$ is independent of $h$ and $\Delta t$. By applying Young's inequality and the Cauchy-Schwarz inequality, we derive the following bound for $E_{11}$:
\begin{equation*} 
\begin{aligned}  
 E_{11} 
 = & \sum_{n=0}^{N-1}\Delta tc_f(\bm K \nabla p^{n+1}\cdot\nabla T^{n+1}- \bm K \nabla p^{n+1}_h\cdot\nabla T^{n+1}_h,e_T^{h,n+1})\\
=& c_f\Delta t \sum_{n=0}^{N-1} 
   (\bm K( \nabla p^{n+1}-\nabla p^{n+1}_h)\cdot\nabla T^{n+1}    
 + \bm K \nabla p^{n+1}_h\cdot(\nabla T^{n+1}-  \nabla T^{n+1}_h)        
           ,e_T^{h,n+1}  )\\
=& c_f\Delta t \sum_{n=0}^{N-1} 
   (\bm K \nabla e_p^{h,n+1}\cdot\nabla T^{n+1}    
 +  \nabla p^{n+1}_h\cdot\bm K\nabla e_T^{h,n+1}        
           ,e_T^{h,n+1}  )\\
\le& c_f\Delta t \sum_{n=0}^{N-1}\Big(  
      \|\bm K\nabla e_p^{h,n+1}\|_{L^2(\Omega)}^2   \|\nabla T^{n+1}\|_{L^\infty(\Omega)}^2     
   +     \| \nabla p^{n+1}_h\|_{L^\infty(\Omega)}^2  \|\bm K\nabla e_T^{h,n+1}\|_{L^2(\Omega)}^2     \\   
    &\quad \quad \quad \quad  + \tfrac{1}{2} \| e_T^{h,n+1} \|_{L^2(\Omega)}^2   \Big)\\
\le&  c_f\Delta t\Big( C_u^2 k_M\sum_{n=0}^{N-1} \| \nabla e_p^{h,n+1} \|_{L^2(\Omega)}^2 
             +   C_u^2k_M \sum_{n=0}^{N-1}  \| \nabla e_T^{h,n+1} \|_{L^2(\Omega)}^2   \\
   &\quad \quad      +\tfrac{1}{2}\sum_{n=0}^{N-1}  \|  e_T^{h,n+1} \|_{L^2(\Omega)}^2 
\Big).\\
\end{aligned}
\end{equation*} 
By integrating equations \eqref{eq: LHS=Ei} and \eqref{ineq: LHS ge etc.}  with the previously established bounds for $E_1$ to $E_{11}$ we derive the following:
 \begin{equation}  \label{eq: before Gronwall}
\begin{aligned}  
 & 2\mu \|\bm\varepsilon(e_{\bm u}^{h,N} )\|^2_{L^2(\Omega)}
        +\tfrac{1}{\lambda} \|\ e_{\xi}^{h,N}
                    -\alpha e_p^{h,N}-\beta e_T^{h,N}\|^2_{L^2(\Omega)}    
  +(c_0-b_0) \| e_p^{h,N} \|^2_{L^2(\Omega)} \\
  & +(a_0-b_0) \| e_T^{h,N} \|^2_{L^2(\Omega)}  
   +b_0  \| e_p^{h,N}-e_T^{h,N} \|^2_{L^2(\Omega)}  
             \\ 
  &  +\Delta t\sum_{n=0}^{N-1}\Big( 
          k_m\|\nabla e_p^{h,n+1}\|^2_{L^2(\Omega)} + \theta_m\|\nabla e_T^{h,n+1}\|^2_{L^2(\Omega)} 
       \Big) \\
\le&\mathbb C_1
\Big(
 h^{2s+2}  \int_{0}^{\tau}\| \xi_t  \|_{H^{k}(\Omega)}^2   
    +h^{2l+2} 
    \int_{0}^{\tau}\| p_t  \|_{H^{l+1}(\Omega)}^2
   + h^{2l+2} \int_{0}^{\tau}\| T_t  \|_{H^{l+1}(\Omega)}^2    \\
&+ \Delta t^2  \int_{0}^{\tau} \left\| \bm u_{tt} \right\|_{H^1(\Omega)}^2
             +\Delta t^2 \int_{0}^{\tau}\| \xi_{tt}  \|_{L^2(\Omega)}^2
             + \Delta t^2 \int_{0}^{\tau}\| p_{tt}  \|_{L^2(\Omega)}^2\\
        &     +\Delta t^2 \int_{0}^\tau \| T_{tt}  \|_{L^2(\Omega)}^2 \Big )
 \quad\quad\quad+ {\mathbb C_2}\Big(  \Delta t \sum_{n=0}^{N-1}\| e_{\xi}^{h,n+1} \|_{L^2(\Omega)}^2\\
& + \Delta t \sum_{n=0}^{N-1} \| e_{p}^{h,n+1} \|_{L^2(\Omega)}^2
 + \Delta t \sum_{n=0}^{N-1} \| e_{T}^{h,n+1} \|_{L^2(\Omega)}^2  
 + \Delta t \sum_{n=0}^{N-1} \| \nabla e_p^{h,n+1} \|_{L^2(\Omega)}^2 \\ 
&\quad\quad     
      +\Delta t\sum_{n=0}^{N-1}  \| \nabla e_T^{h,n+1} \|_{L^2(\Omega)}^2    \Big) ,\\ 
\le&\mathbb C_1
\Big(
 h^{2s+2}  \int_{0}^{\tau}\| \xi_t  \|_{H^{k}(\Omega)}^2   
    +h^{2l+2} 
    \int_{0}^{\tau}\| p_t  \|_{H^{l+1}(\Omega)}^2
   + h^{2l+2} \int_{0}^{\tau}\| T_t  \|_{H^{l+1}(\Omega)}^2    \\
&+ \Delta t^2  \int_{0}^{\tau} \left\| \bm u_{tt} \right\|_{H^1(\Omega)}^2
             +\Delta t^2 \int_{0}^{\tau}\| \xi_{tt}  \|_{L^2(\Omega)}^2
             + \Delta t^2 \int_{0}^{\tau}\| p_{tt}  \|_{L^2(\Omega)}^2\\
        &     +\Delta t^2 \int_{0}^\tau \| T_{tt}  \|_{L^2(\Omega)}^2
               \Big )
    \quad\quad     + \mathbb C_2 \Delta t \Big(  
    \sum_{n=0}^{N-1} 2\mu\| \bm\varepsilon(e_{\bm u}^{h,n+1}) \|_{L^2(\Omega)}^2   \\
&
+  \sum_{n=0}^{N-1} \frac1\lambda
                 \| e_{\xi}^{h,n+1}-\alpha e_p^{h,N}-\beta e_T^{h,N} \|_{L^2(\Omega)}^2  
                 +  \sum_{n=0}^{N-1}(c_0-b_0) \| e_{p}^{h,n+1} \|_{L^2(\Omega)}^2 \\
& 
  +   \sum_{n=0}^{N-1}(a_0-b_0) \| e_{T}^{h,n+1} \|_{L^2(\Omega)}^2  
  +   \sum_{n=0}^{N-1}b_0 \| e_{p}^{h,n+1}-e_{T}^{h,n+1} \|_{L^2(\Omega)}^2  
\Big ),
\end{aligned}
\end{equation}   
where $\mathbb C_1$ and $\mathbb C_2$ denotes a constant that is independent of $h$ and $\Delta t$.
By the integral form of the Grönwall inequality, \eqref{eq: before Gronwall} leads to  
 \begin{equation*}  
\begin{aligned}  
 & 2\mu \|\bm\varepsilon(e_{\bm u}^{h,N} )\|^2_{L^2(\Omega)}
                       +\tfrac{1}{\lambda} \|\ e_{\xi}^{h,N}-\alpha e_p^{h,N}-\beta 
                      e_T^{h,N}\|^2_{L^2(\Omega)}    
  +(c_0-b_0) \| e_p^{h,N} \|^2_{L^2(\Omega)} \\
  & +(a_0-b_0) \| e_T^{h,N} \|^2_{L^2(\Omega)}  
   +b_0  \| e_p^{h,N}-e_T^{h,N} \|^2_{L^2(\Omega)}  
             \\ 
  &  +\tfrac12\Delta t\sum_{n=0}^{N-1}\Big( 
          k_m \|\nabla e_p^{h,n+1} \|^2_{L^2(\Omega)} + \theta_m \|\nabla e_T^{h,n+1} \|^2_{L^2(\Omega)}
       \Big) \\
\le&\tilde C
\Big(
 h^{2s+2}  \int_{0}^{\tau}\| \xi_t  \|_{H^{k}(\Omega)}^2   
    +h^{2l+2} 
    \int_{0}^{\tau}\| p_t  \|_{H^{l+1}(\Omega)}^2
   + h^{2l+2} \int_{0}^{\tau}\| T_t  \|_{H^{l+1}(\Omega)}^2    \\
&+ \Delta t^2  \int_{0}^{\tau} \left\| \bm u_{tt} \right\|_{H^1(\Omega)}^2
             +\Delta t^2 \int_{0}^{\tau}\| \xi_{tt}  \|_{L^2(\Omega)}^2
             + \Delta t^2 \int_{0}^{\tau}\| p_{tt}  \|_{L^2(\Omega)}^2\\
        &     +\Delta t^2 \int_{0}^\tau \| T_{tt}  \|_{L^2(\Omega)}^2
             \Big ).
\\ 
\end{aligned}
\end{equation*}  
Taking $\bm v_h=e_{\bm u}^{h,n+1}$ and $\phi_h=e_{\xi}^{h,n+1}$ 
in \eqref{eq: evun+1-evxin+1=0}-\eqref{eq: -eun+1-exin+1+pn+1+Tn+1=0} and summing these two equations, we have the estimate 
$\|e_{\xi}^{h,N}\|_{L^2(\Omega)}\le\alpha\|e_p^N\|_{L^2(\Omega)}
                                       +\beta\|e_T^N \|_{L^2(\Omega)}$ .

Then, the desired result is obtained from the above estimate and an application of the triangle inequality on \eqref{eq: e=eI+eh}. Moreover, if $\lambda\ge1$, then the constant $\mathbb C_1$ and $\mathbb C_2$ remain bounded as $\lambda\to\infty$, because $\lambda$ enters the estimate only through $1/\lambda$. Hence \eqref{eq: non error estimate} is uniform with respect to the  Lamé parameter in the nearly incompressible limit.
\end{proof}

\begin{lemma}\label{lemma: poincare}
For the bounded domain $\Omega$, there exists a constant \(C_P > 0\) such that for every $v\in H_0^1(\Omega)$ the following Poincare inequality \cite{girault2012finite} is satisfied:
\[
       C_P\| v\|_{L^2(\Omega)} \leq \| \nabla v\|_{L^2(\Omega)}  .
\]
\end{lemma} 
\begin{lemma}\label{lemma: inverse inequ}
There exists a constant \(C_{inv} > 0\) such that the following inverse inequality is satisfied:  
\[
\|\bm v\|_{H^1(\Omega)} \leq C_{inv} h^{-1} \|\bm v\|_{L^2(\Omega)}, \quad \forall \bm v \in \bm V_h,
\]
 as established in \cite{ciarlet2002finite}.
\end{lemma}

To proceed with the analysis, in the remaining part of the article, we
let $\eta_{\bm u}^i,\eta_{\xi}^i,\eta_{p}^i,\eta_{T}^i$ be differences between the solutions at the iteration $i$ of iterative algorithms and
the solutions to problem \eqref{TP_Model_dis_a}-\eqref{TP_Model_dis_d}, i.e.
\begin{equation}\label{diff_iter_real}
\begin{aligned} 
     (\eta_{\bm u}^i,\eta_{\xi}^i,\eta_{p}^i,\eta_{T}^i):=
     (\bm u_h^{n},\xi_h^{n},p_h^{n},T_h^{n})
  -(\bm u_h^{n,i},\xi_h^{n,i},p_h^{n,i},T_h^{n,i}).
\end{aligned}
\end{equation}

Since the convergence proofs within each pair are similar, we only prove the most intractable case here: the MF-H algorithm. The remaining results follow analogously.
 
\begin{proof}[\textbf{The proof of the convergence of MF-H algorithm: }]

Assuming the \textbf{Assumptions} (A1)-(A4) hold.
Let $(\bm u_h^n,\xi_h^n,p_h^n,T_h^n)$ be the solution of the coupled problem \eqref{eq:coupled alg}. 
and $(\bm u_h^{n,i},\xi_h^{n,i},p_h^{n,i},T_h^{n,i})$ be the solution of the iteratively decoupled problem \eqref{alg: MF_H 1}-\eqref{alg: MF_H 2}.
We next prove the solution sequence $\{(\bm u_h^{n,i},\xi_h^{n,i},p_h^{n,i},T_h^{n,i})\}_{i=1}^{\infty}$ generated by the iteratively decoupled algorithm converges to the solution $(\bm u_h^n,\xi_h^n,p_h^n,T_h^n)$ of the coupled problem. 
Subtracting equation \eqref{alg: MF_H 1} -\eqref{alg: MF_H 2}  from equation \eqref{TP_Model_dis_a}-\eqref{TP_Model_dis_d} respectively, we obtain 
\begin{equation}\label{eq: var_N1c}
\begin{aligned} 
2\mu(\bm\varepsilon(\eta_{\bm u}^i),\bm\varepsilon(\bm v_h)) - (\nabla\cdot\bm v_h,\eta_{\xi}^i)=&0 ,   \\
\end{aligned}
\end{equation}
\begin{equation}\label{eq: var_N1d}
\begin{aligned} 
  (\nabla\cdot \eta_{\bm u}^{i},\phi_h)+\tfrac{1}{\lambda}(\eta_{\xi}^i,\phi_h) 
     -\tfrac{\alpha}{\lambda}(\eta_{p}^{i},\phi_h)+\rho_\xi(\eta_{\xi}^i-\eta_{\xi}^{i-1},  \phi_h)=& 
       \tfrac{\beta}{\lambda}(\eta_{T}^{i-1},\phi_h),\\ 
\end{aligned}
\end{equation} 
\begin{equation}\label{eq: var_N1a}
\begin{aligned} 
       -\tfrac\alpha\lambda(\eta_{\xi}^{i},q_h) 
       +c_\alpha(\eta_{p}^i,q_h)  
                &+ \Delta t(\bm K\nabla \eta_{p}^i,\nabla q_h)  
                 + \rho_p(\eta_p^{i}-\eta_p^{i-1},q_h)                      
             =     
                 - c_{\alpha\beta}(\eta_{T}^{i-1},q_h),\\
\end{aligned}
\end{equation}
\begin{equation}\label{eq: var_N1b}
\begin{aligned} 
      c_\beta(\eta_{T}^i,S_h) 
               +\Delta t(\bm\Theta\nabla \eta_{T}^i,\nabla S_h) 
         &+ \rho_T(\eta_T^{i}-\eta_T^{i-1},S_h)      \\ 
               = &\Delta t c_f(\mathcal L^{n,i-1},S_h) + \tfrac\beta\lambda(\eta_{\xi}^{i-1},S_h) 
            - c_{\alpha\beta}(\eta_{p}^{i-1},S_h)   ,\\
\end{aligned}
\end{equation}   
where $ \mathcal L^{n,i-1}  =  \bm K\nabla p^n_h\cdot\nabla T^n_h -\bm K\nabla p_h^{n,i-1}\cdot\nabla T_h^{n,i-1} $.
Taking $q_h=\eta_{p}^i$, $S_h=\eta_{T}^i$, $\bm v_h=\eta_{\bm u}^{i}$ and $\phi_h=\eta_{\xi}^i$ in equation \eqref{eq: var_N1c},  \eqref{eq: var_N1d}, \eqref{eq: var_N1a}  and \eqref{eq: var_N1b} and summing up the equations,  we have   
\begin{equation}\label{eq: sum H_F_M}  
\begin{aligned}   
  &        2\mu(\bm\varepsilon(\eta_{\bm u}^i),\bm\varepsilon(\eta_{\bm u}^{i}))  
                           +\tfrac{1}{\lambda}(\eta_{\xi}^i,\eta_{\xi}^i)     
                           -\tfrac{\alpha}{\lambda}(\eta_{p}^{i},\eta_{\xi}^i)
                            +\rho_\xi(\eta_{\xi}^i-\eta_{\xi}^{i-1},\eta_{\xi}^i)\\
  &          -  \tfrac\alpha\lambda( \eta_{\xi}^{i},\eta_{p}^i  )  
             + c_\alpha\| \eta_{p}^i\|_{L^2(\Omega)}^2 
                  + \Delta t(\bm K\nabla \eta_{p}^i,\nabla \eta_{p}^i)   
            + \rho_p(\eta_{p}^i-\eta_{p}^{i-1},\eta_{p}^i)            
            \\
   &    +   c_\beta\Vert \eta_{T}^i\|_{L^2(\Omega)}^2 
         + \Delta t(\bm\Theta\nabla \eta_{T}^i,\nabla \eta_{T}^i )
        + \rho_T(\eta_{T}^i-\eta_{T}^{i-1},\eta_{T}^i)   \\   
 = &     \tfrac{\beta}{\lambda}(\eta_{T}^{i-1},\eta_{\xi}^i) 
         -  c_{\alpha\beta}( \eta_{T}^{i-1}, \eta_{p}^i)  
       +\Delta t c_f(\mathcal L^{n,i-1},\eta_{T}^i)
        + \tfrac\beta\lambda( \eta_{\xi}^{i-1}, \eta_{T}^i)   
       - c_{\alpha\beta}( \eta_{p}^{i-1}, \eta_{T}^i) 
    .\\  
\end{aligned}
\end{equation} 
Since the sign of $c_{\alpha\beta}$ cannot be determined, we introduce the sign function sgn$(\cdot)$ to unify the two cases $c_{\alpha\beta}\ge0$ and $c_{\alpha\beta}<0$. Then, following equation \eqref{eq: sum H_F_M}, we obtain:
\begin{equation*}\label{eq: rearrange H_F_M}  
\begin{aligned}   
  &     c_\alpha\| \eta_{p}^i\|_{L^2(\Omega)}^2              
     +   c_\beta\Vert \eta_{T}^i\|_{L^2(\Omega)}^2 
       +2c_{\alpha\beta}(\eta_{p}^i, \eta_{T}^i)
       + \tfrac\alpha\lambda (\eta_{\xi}^i-\eta_{p}^i-\eta_{\xi}^{i-1}+\eta_{p}^{i-1},\eta_{\xi}^i-\eta_{p}^i) \\
     &            + \Delta t(\bm\Theta\nabla \eta_{T}^i,\nabla \eta_{T}^i )     
       + \Delta t(\bm K\nabla \eta_{p}^i,\nabla \eta_{p}^i)    
  + 2\mu(\bm\varepsilon(\eta_{\bm u}^i),\bm\varepsilon(\eta_{\bm u}^{i}))  
     +\tfrac{1}{\lambda}(\eta_{\xi}^i,\eta_{\xi}^i)    \\ 
   &     + (\rho_p - |c_{\alpha\beta}|-\tfrac\alpha\lambda)(\eta_{p}^i-\eta_{p}^{i-1},\eta_{p}^i) 
        + (\rho_T - |c_{\alpha\beta}|)(\eta_{T}^i-\eta_{T}^{i-1},\eta_{T}^i)  
        +(\rho_\xi-\tfrac\alpha\lambda)(\eta_{\xi}^i -\eta_{\xi}^{i-1} ,\eta_{\xi}^i )  \\   
&  + |c_{\alpha\beta}|( (\eta_{p}^i- \text{sgn}(c_{\alpha\beta})\eta_{T}^i)
-(\eta_{p}^{i-1}- \text{sgn}(c_{\alpha\beta})\eta_{T}^{i-1}) , (\eta_{p}^i-\text{sgn}(c_{\alpha\beta})\eta_{T}^i))    \\
=  &     c_\alpha\| \eta_{p}^i\|_{L^2(\Omega)}^2             
     +   c_\beta\Vert \eta_{T}^i\|_{L^2(\Omega)}^2 
       +2c_{\alpha\beta}(\eta_{p}^i, \eta_{T}^i) 
        +\rho_\xi(\eta_{\xi}^i-\eta_{\xi}^{i-1},\eta_{\xi}^i)\\
     &            + \Delta t(\bm\Theta\nabla \eta_{T}^i,\nabla \eta_{T}^i )     
       + \Delta t(\bm K\nabla \eta_{p}^i,\nabla \eta_{p}^i)    
  + 2\mu(\bm\varepsilon(\eta_{\bm u}^i),\bm\varepsilon(\eta_{\bm u}^{i}))  
     +\tfrac{1}{\lambda}(\eta_{\xi}^i,\eta_{\xi}^i)    \\ 
   &   + \rho_p(\eta_{p}^i-\eta_{p}^{i-1},\eta_{p}^i) 
        + \rho_T(\eta_{T}^i-\eta_{T}^{i-1},\eta_{T}^i)   \\   
&  + c_{\alpha\beta}( \eta_{p}^{i-1}, \eta_{T}^i)+ c_{\alpha\beta}( \eta_{T}^{i-1}, \eta_{p}^i)  
     -2c_{\alpha\beta}(\eta_{p}^i, \eta_{T}^i)  \\ 
 &   + \tfrac\alpha\lambda( \eta_{\xi}^{i-1},\eta_{p}^i )
    + \tfrac\alpha\lambda( \eta_p^{i-1},\eta_{\xi}^i )
     -2\tfrac\alpha\lambda( \eta_p^{i},\eta_{\xi}^i )\\
=&  \Delta t c_f(\mathcal L^{n,i-1},\eta_{T}^i)
    + \tfrac1\lambda( \eta_{\xi}^{i-1},\alpha\eta_{p}^i+\beta\eta_{T}^i  )
    + \tfrac{1}{\lambda}(\alpha\eta_{p}^{i-1}+\beta\eta_{T}^{i-1},\eta_{\xi}^i),\\ 
\end{aligned}
\end{equation*} 
yielding   
\begin{equation}\label{eq: i le i-1 H_F_M}  
\begin{aligned}    
& (c_0-b_0)\|\eta_{p}^i\|_{L^2(\Omega)}^2+(a_0-b_0)\|\eta_{T}^i\|_{L^2(\Omega)}^2
  +\tfrac1{2\lambda}\|\alpha \eta_{p}^i+\beta \eta_{T}^i \|_{L^2(\Omega)}^2\\
&+2\mu\|\bm\varepsilon(\eta_{\bm u}^i)\|_{L^2(\Omega)}^2
 + \tfrac1{2\lambda} \|\eta_{\xi}^i\|_{L^2(\Omega)}^2 
 + \Delta t k_m\|\nabla\eta_{p}^i\|_{L^2(\Omega)}^2    
 + \Delta t \theta_m\|\nabla\eta_{T}^i\|_{L^2(\Omega)}^2          \\ 
 &    + \tfrac12\Big(  \rho_p-\ |c_{\alpha\beta}|- \tfrac\alpha\lambda \Big)\| \eta_{p}^i\|_{L^2(\Omega)}^2        
 +   \tfrac12\Big(  \rho_T-|c_{\alpha\beta}|   \Big)\| \eta_{T}^i\|_{L^2(\Omega)}^2 
   + \tfrac12\Big(   \rho_\xi- \tfrac\alpha\lambda \Big)\| 
                                       \eta_{\xi}^{i}\|_{L^2(\Omega)}^2  \\  
   &  + \tfrac12|c_{\alpha\beta}| \| \eta_{p}^i- \text{sgn}(c_{\alpha\beta})\eta_{T}^i \|_{L^2(\Omega)}^2  
     + \tfrac12\tfrac\alpha\lambda
               \| \eta_{\xi}^{i} -  \eta_{p}^{i} \|_{L^2(\Omega)}^2\\  
\le  &   \Delta t c_f(\mathcal L^{n,i-1},\eta_{T}^i) 
     + \tfrac{1}{2\lambda}\| \eta_{\xi}^{i-1}\|_{L^2(\Omega)}^2 
     +   \tfrac{1}{2\lambda} \|\alpha\eta_{p}^{i-1} 
                                 +  \beta\eta_{T}^{i-1}\|_{L^2(\Omega)}^2 \\
& +\tfrac12 \Big(   \rho_p- |c_{\alpha\beta}| - \tfrac\alpha\lambda \Big)\| 
                                         \eta_{p}^{i-1}\|_{L^2(\Omega)}^2   
  +\tfrac12 \Big(    \rho_T- |c_{\alpha\beta}| \Big)\| 
                                       \eta_{T}^{i-1}\|_{L^2(\Omega)}^2
  +\tfrac12 \Big(    \rho_\xi- \tfrac\alpha\lambda \Big)\| 
                                       \eta_{\xi}^{i-1}\|_{L^2(\Omega)}^2 
   \\
& + \tfrac12|c_{\alpha\beta}| 
               \| \eta_{p}^{i-1} - \text{sgn}(c_{\alpha\beta})\eta_{T}^{i-1} \|_{L^2(\Omega)}^2
  + \tfrac12\tfrac\alpha\lambda
               \| \eta_{\xi}^{i-1} -  \eta_{p}^{i-1} \|_{L^2(\Omega)}^2.   \\ 
\end{aligned}
\end{equation} 

For the nonlinear term, we still assume that the $L^\infty$-norm of $\nabla p_h^{n,i}$, $\nabla p_h^{n}$, $\nabla T_h^{n,i}$ and, $\nabla T_h^n$ are bounded above on domain $\Omega$. By the Poincaré inequality and the embedding theorem, the $L^2$-norm of $\nabla p_h^{n,i}$, $\nabla p_h^{n}$, $\nabla T_h^{n,i}$ and, $\nabla T_h^n$ admit an upper bound. We uniformly denote all these upper bounds by $C_u$. By Lemma \ref{lemma: inverse inequ}, we obtain 
\begin{equation}\label{eq: estimate L}  
\begin{aligned}   
    (\mathcal L^{n,i-1},\eta_{T}^i)  
 = &(\bm K\nabla p^n_h\cdot\nabla T^n_h -\bm K\nabla p_h^{n,i-1}\cdot\nabla T_h^{n,i-1} ,\eta_{T}^i)\\
=&((\bm K\nabla p^n_h -\bm K\nabla p_h^{n,i-1} )\cdot\nabla T^n_h+
     \bm K\nabla p_h^{n,i-1}\cdot(\nabla T^{n}_h-\nabla T_h^{n,i-1} ),\eta_{T}^i)\\
\le&
\|\bm K\nabla p^n_h -\bm K\nabla p_h^{n,i-1}\|_{L^2(\Omega)}^2 \| \nabla T^{n}_h \|_{L^\infty(\Omega)}^2 \\
&+ \|\bm K\nabla p_h^{n,i-1}\|_{L^\infty(\Omega)}^2 \|\nabla T^{n}_h-\nabla T_h^{n,i-1}\|_{L^2(\Omega)}^2
+\tfrac12\|\eta_{T}^i\|_{L^2(\Omega)}^2\\
\le&
    k_M\| \nabla \eta_{p}^{i-1} \|_{L^2(\Omega)}^2 \| \nabla T^{n}_h \|_{L^\infty(\Omega)}^2 \\
&+ k_M\|\nabla p_h^{n,i-1}\|_{L^\infty(\Omega)}^2 \|\nabla\eta_{T}^{i-1}\|_{L^2(\Omega)}^2
+\tfrac12\|\eta_{T}^i\|_{L^2(\Omega)}^2\\
\le&
   k_MC_u^2C_{inv}^2h^{-2} \| \eta_{p}^{i-1} \|_{L^2(\Omega)}^2   
 + k_MC_u^2C_{inv}^2h^{-2}  \| \eta_{T}^{i-1}\|_{L^2(\Omega)}^2
+\tfrac12\|\eta_{T}^i\|_{L^2(\Omega)}^2.\\
\end{aligned}
\end{equation}
Combining equation \eqref{eq: i le i-1 H_F_M}  , \eqref{eq: estimate L}, and Lemma \ref{lemma: inverse inequ}, we obtain
\begin{equation}\label{eq: i le i-1}  
\begin{aligned}   & (c_0-b_0)\|\eta_{p}^i\|_{L^2(\Omega)}^2+(a_0-b_0)\|\eta_{T}^i\|_{L^2(\Omega)}^2
  +\tfrac1{2\lambda}\|\alpha \eta_{p}^i+\beta \eta_{T}^i \|_{L^2(\Omega)}^2\\
&+2\mu\|\bm\varepsilon(\eta_{\bm u}^i)\|_{L^2(\Omega)}^2
 + \tfrac1{2\lambda} \|\eta_{\xi}^i\|_{L^2(\Omega)}^2 
 + \Delta t k_mC_P^2\|\eta_{p}^i\|_{L^2(\Omega)}^2    
 + (\Delta t\theta_mC_P^2-\tfrac12\Delta t c_f)\|\eta_{T}^i\|_{L^2(\Omega)}^2          \\ 
 &   
  + \tfrac12\Big(   \rho_p-|c_{\alpha\beta}|  -\tfrac\alpha\lambda \Big)\| \eta_{p}^i\|_{L^2(\Omega)}^2        
  +   \tfrac12\Big(    \rho_T- |c_{\alpha\beta}| \Big)\| \eta_{T}^i\|_{L^2(\Omega)}^2 
  + \tfrac12\Big(   \rho_\xi- \tfrac\alpha\lambda \Big)\| \eta_{\xi}^{i}\|_{L^2(\Omega)}^2 \\  
   &  
    + \tfrac12|c_{\alpha\beta}| \| \eta_{p}^i- \text{sgn}(c_{\alpha\beta})\eta_{T}^i \|_{L^2(\Omega)}^2 
   + \tfrac12\tfrac\alpha\lambda
               \| \eta_{\xi}^{i} -  \eta_{p}^{i} \|_{L^2(\Omega)}^2 \\  
\le&  \Delta tc_fk_MC_u^2C_{inv}^2h^{-2}  \|\eta_{p}^{i-1}\|_{L^2(\Omega)}^2 
   +  \Delta tc_fk_MC_u^2C_{inv}^2h^{-2}  \|\eta_{T}^{i-1}\|_{L^2(\Omega)}^2
    \\
&     + \tfrac{1}{2\lambda}\| \eta_{\xi}^{i-1}\|_{L^2(\Omega)}^2 
     +   \tfrac{1}{2\lambda} \|\alpha\eta_{p}^{i-1} 
                                 +  \beta\eta_{T}^{i-1}\|_{L^2(\Omega)}^2 \\
& + \tfrac12\Big( \rho_p-|c_{\alpha\beta}| -\tfrac\alpha\lambda \Big)\| 
                                         \eta_{p}^{i-1}\|_{L^2(\Omega)}^2   
  +\tfrac12 \Big( \rho_T- |c_{\alpha\beta}| \Big)\| 
                                       \eta_{T}^{i-1}\|_{L^2(\Omega)}^2  
  + \tfrac12\Big(   \rho_\xi- \tfrac\alpha\lambda \Big)\| 
                                       \eta_{\xi}^{i-1}\|_{L^2(\Omega)}^2 
   \\
& + \tfrac12|c_{\alpha\beta}| 
               \| \eta_{p}^{i-1} - \text{sgn}(c_{\alpha\beta})\eta_{T}^{i-1}                                     \|_{L^2(\Omega)}^2
  + \tfrac12\tfrac\alpha\lambda
               \| \eta_{\xi}^{i-1} -  \eta_{p}^{i-1} \|_{L^2(\Omega)}^2. 
\end{aligned}
\end{equation}
Consequently, if $c_f=0$, it obviously‌ to show $\eta_{\bm u}^i$, $\eta_\xi^i$, $\eta_p^i$, $\eta_p^i$, $\eta_T^i\to 0$, otherwise, we can rewrite \eqref{eq: i le i-1} as 
\begin{equation*} 
  (1+\delta)\mathcal A^{i} + \mathcal B^i< \mathcal A^{i-1} + \mathcal B^{i-1}, 
\end{equation*} 
where 
\begin{equation*} 
  \mathcal A^{i} =
         \Delta t c_fk_MC_u^2C_{inv}^2h^{-2}    \| \eta_{p}^{i}\|_{L^2(\Omega)}^2   
      +  \Delta t c_fk_MC_u^2C_{inv}^2h^{-2}    \| \eta_{T}^{i}\|_{L^2(\Omega)}^2  , 
\end{equation*} 
\begin{equation*} 
\begin{aligned} 
   \mathcal B^i=&\tfrac{1}{2\lambda}\| \eta_{\xi}^{i}\|_{L^2(\Omega)}^2
   +   \tfrac{1}{2\lambda} \|\alpha\eta_{p}^{i} +  \beta\eta_{T}^{i}\|_{L^2(\Omega)}^2 
     + \tfrac12\Big(  \rho_p-|c_{\alpha\beta}| -\tfrac\alpha\lambda \Big)\| \eta_{p}^{i}\|_{L^2(\Omega)}^2   
  +    \tfrac12\Big(  \rho_T- |c_{\alpha\beta}|     \Big)\| \eta_{T}^{i}\|_{L^2(\Omega)}^2   \\
    &  + \tfrac12\Big(   \rho_\xi- \tfrac\alpha\lambda \Big)\| \eta_{\xi}^{i}\|_{L^2(\Omega)}^2
     + \tfrac12|c_{\alpha\beta}| \| \eta_{p}^{i}- \text{sgn}(c_{\alpha\beta})\eta_{T}^{i} \|_{L^2(\Omega)}^2 
     + \tfrac12\tfrac\alpha\lambda
               \| \eta_{\xi}^{i} -  \eta_{p}^{i} \|_{L^2(\Omega)}^2,   \\ 
\end{aligned}
\end{equation*} 
and 
\begin{equation*} 
\begin{aligned} 
\mathcal \delta =\min\Big\{ &
        \tfrac{( c_0-b_0 )+\Delta tk_mC_p^2- \Delta tc_fk_MC_u^2C_{inv}^2h^{-2}} 
        {\Delta t c_fk_MC_u^2C_{inv}^2h^{-2}},
     \tfrac{(a_0-b_0) - \tfrac12\Delta t c_f + \Delta t \theta_mC_p^2
        -  \Delta tc_fk_MC_u^2C_{inv}^2h^{-2}}
        {\Delta t c_fk_MC_u^2C_{inv}^2h^{-2}}
      \Big\}.
\end{aligned}
\end{equation*} 
If we take a small enough $\Delta t$
\begin{equation*} 
\begin{aligned} 
 \Delta t<\min\{
          \frac{( c_0-b_0 )h^2}{   c_fk_MC_u^2C_{inv}^2 - k_mC_p^2h^2 },
          ~~
          \frac{ (a_0-b_0 )h^2 }{ c_f  k_MC_u^2C_{inv}^2
                                  +\tfrac12 c_f h^2
                                - \theta_mC_p^2h^2}    \},
\end{aligned}                                   
\end{equation*} 
and large enough $\rho_p,\rho_T$
\begin{equation*} 
\begin{aligned} 
 \rho_p\ge |c_{\alpha\beta}| + \tfrac\alpha\lambda,~~~~\rho_T\ge |c_{\alpha\beta}|,~~~~\rho_\xi\ge\tfrac\alpha\lambda
\end{aligned}                                   
\end{equation*} 
such that $\delta,\mathcal B^i>0$, then, we have $\mathcal{A}^i \to 0$ as $i \to \infty$, which implies $\eta_p^i$, $\eta_T^i\to 0$ as $i \to \infty$. 
Furthermore,  by equation \eqref{eq: var_N1c}-\eqref{eq: var_N1d}
we conclude that  $\eta_{\xi}^i$, $\eta_{\bm u}^i \to 0$ as $i \to \infty$. 

Note that if the priori bound $C_u$ is available such that $k_mC_p^2- c_fk_MC_u^2C_{inv}^2h^{-2}\ge0$ and $ \theta_mC_p^2- \tfrac12 c_f  - c_fk_MC_u^2C_{inv}^2h^{-2}\ge0$ hold true, then $\delta>0$ holds and there is no constraint on the time step. Otherwise, the upper bound $C_u$ is supposed to be a sufficiently large constant so that 
$$ k_mC_p^2-  c_fk_MC_u^2C_{inv}^2h^{-2}<0 \quad\text{ and }\quad 
\theta_mC_p^2 - \tfrac12 c_f  - c_fk_MC_u^2C_{inv}^2h^{-2}<0,$$ 
thereby ensuring that the time step $\Delta t$ remains less than a positive value.  

\end{proof}

\section{Numerical experiments}\label{sec: experiments}
For the discrete model, we employ $\bm{\mathcal P}_k$ Lagrange finite elements for $\bm{u}$, ${\mathcal P}_{s}$ Lagrange finite elements for $\xi$, along with $\mathcal P_l$ Lagrange finite elements for both the fluid pressure $p$ and temperature $T$. The finite element spaces defined on mesh \(\mathcal{T}_h\) are as follows.
\[
\begin{aligned}
&\bm{V}_h = \left\{ \bm{v}_h \in [H_{0,\Gamma
_d}^1(\Omega)]^d \cap \bm{C}^0(\bar{\Omega}) : \bm{v}_h|_K \in \bm{\mathcal P}_k(K), \, \forall K \in \mathcal{T}_h \right\}, \\
&Q_h = \left\{ \phi_h \in L^2(\Omega) \cap C^0(\bar{\Omega}) : \phi_h|_K \in \mathcal{P}_{s}(K), \, \forall K \in \mathcal{T}_h \right\}, \\
&W_h = \left\{ q_h \in H_0^1(\Omega) \cap C^0(\bar{\Omega}) : q_h|_K \in \mathcal{P}_l(K), \, \forall K \in \mathcal{T}_h \right\}.
\end{aligned}
\]

This section presents a series of numerical experiments aimed at evaluating the computational accuracy and efficiency of the algorithms proposed in Section \ref{sec: coupled algorithm}. The primary objective is to assess the performance of the different parallel algorithms under various physical parameter settings.
All numerical experiments are conducted using the following benchmark problem, and all computations are performed using the FreeFEM++ software \cite{Hecht2012NewDI}.

In this section, we adopt the finite element spaces $\bm V_h, Q_h, W_h$ defined above with degrees $k=2,s=1,l=1$ satisfying discrete inf-sup condition over the domain $\Omega = [0,1] \times [0,1]$. Next, we consider the following example.

\textbf{Example 1:}
In this example, we appropriately select the body force $\bm{f}$, mass source $g$, and heat source $h$ to ensure that the exact solution is given by:
\begin{equation} \label{Ex: example1}
\begin{aligned}
\bm{u}(x,y,t) &= e^{-t} \begin{pmatrix}
\sin(2\pi y)(\cos(2\pi x) - 1) + \tfrac{1}{\mu + \lambda} \sin(\pi x) \sin(\pi y) \\
\sin(2\pi x)(1 - \cos(2\pi y)) + \tfrac{1}{\mu + \lambda} \sin(\pi x) \sin(\pi y)
\end{pmatrix}, \\
p(x,y,t) &= e^{-t} \sin(\pi x) \sin(\pi y), \\
T(x,y,t) &= e^{-t} \sin(\pi x) \sin(\pi y).
\end{aligned}
\end{equation}
We impose Dirichlet boundary conditions $\bm{u} = \bm{0}$ on the boundary segment $\Gamma_d = \{(x,y)\,|\, 0 \le y \le 1,\, x = 0 \text{, } 1\}$, and Neumann boundary conditions on $\Gamma_n = \{(x,y)\,|\, 0 \le x \le 1,\, y = 0 \text{, } 1\}$. For the pressure $p$ and temperature $T$, homogeneous Dirichlet boundary conditions are prescribed on the entire boundary $\partial\Omega$.  The initial conditions for $\bm{u}^0$, $p^0$, and $T^0$ at time $t = 0$ are given by the expressions in \eqref{Ex: example1}.

In the subsequent numerical experiments, we employ uniform triangular meshes with mesh sizes $h = \tfrac{1}{8}$, $h = \tfrac{1}{16}$, $h = \tfrac{1}{32}$, and $h = \tfrac{1}{64}$. These meshes are generated by iteratively bisecting each triangle through its midpoints, resulting in a refined mesh structure. 
At the terminal time $\tau$, we compute and report the $L^2$-norm errors for the variable $\xi$, as well as the $H^1$-norm errors for the displacement field $\bm{u}$, pressure $p$, and temperature $T$. Additionally, we provide the corresponding convergence rates established in Theorem \ref{th: error estimate} for these errors. 
Unless otherwise stated, the terminal time for the experiments is set to $\tau = 0.001$. The time step size is set to $\Delta t = 10^{-4}$, which requires $10$ time steps to reach the terminal time $\tau$. 
The number of iterations in each time step is also set to $10$.
The numerical results presented in Tables \ref{tab: K Theta 0.1 nu 0.3}-\ref{tab: strong coupling} evaluate the performance and accuracy of the proposed algorithms under various physical parameter settings. 
Table \ref{tab: K Theta 0.1 nu 0.3} displays the results for a compressible thermo-poroelastic material, using the baseline parameters \(E = 1\), \(c_f=1\), \(c_0 = a_0 = 0.2\), \(b_0 = 0.1\), \(\alpha = \beta = 0.1\), \(\bm{K} = \bm{\Theta} = 0.1 \bm{I}\), and a Poisson ratio of \(\nu=0.3\).
The stabilization parameters are set to $\rho_p=\rho_T=\rho_\xi=1$ in this case.

\subsection{Large Lamé parameters}\label{subsec: Large Lame parameters}

To investigate the nearly incompressible regime, we set the Poisson ratio to $\nu = 0.4999999$ while keeping all other parameters at their baseline values and $\rho_p=\rho_T=\rho_\xi=1$. As $\nu$ approaches the theoretical limit of $0.5$ (where $\lambda \to \infty$), the thermo-poroelastic model becomes effectively incompressible. Table \ref{tab: nu 0.49999} shows that the energy-norm errors for all algorithms retain optimal convergence rates under this limiting case, thereby confirming their robustness with respect to $\nu$.

\begin{table}[htbp] \tiny  
\caption{Example 1: The errors and convergence orders with decreasing mesh sizes of
algorithms. $\nu=0.3$. }
\centering
\begin{tabular*}{0.9\textwidth}{@{\extracolsep{\fill}} l ccccccccc @{\extracolsep{\fill}}}
\toprule
  $h$&$\Vert\bm u\Vert_{H^1}$ &rate &$\Vert\xi\Vert_{L^2}$ &rate &$\Vert p\Vert_{H^1}$ &rate&$\Vert T\Vert_{H^1}$ &rate \\ 
\midrule
 MFH &\\
\midrule
$1/8$ & 3.97847e-01 & 0.00 & 2.59030e-02 & 0.00 & 7.32951e-01 & 0.00 & 7.37281e-01 & 0.00 \\
$1/16$ & 1.01965e-01 & 1.96 & 6.26926e-03 & 2.05 & 2.70959e-01 & 1.44 & 2.72271e-01 & 1.44 \\
$1/32$ & 2.56523e-02 & 1.99 & 1.55735e-03 & 2.01 & 1.16401e-01 & 1.22 & 1.16619e-01 & 1.22 \\
$1/64$ & 6.42322e-03 & 2.00 & 3.88807e-04 & 2.00 & 5.54321e-02 & 1.07 & 5.54611e-02 & 1.07 \\
\midrule
Parallel M-FH &\\
\midrule
$1/8$ & 3.97847e-01 & 0.00 & 2.59030e-02 & 0.00 & 7.32954e-01 & 0.00 & 7.37284e-01 & 0.00 \\
$1/16$ & 1.01965e-01 & 1.96 & 6.26926e-03 & 2.05 & 2.70959e-01 & 1.44 & 2.72271e-01 & 1.44 \\
$1/32$ & 2.56523e-02 & 1.99 & 1.55734e-03 & 2.01 & 1.16401e-01 & 1.22 & 1.16618e-01 & 1.22 \\
$1/64$ & 6.42322e-03 & 2.00 & 3.88796e-04 & 2.00 & 5.54317e-02 & 1.07 & 5.54607e-02 & 1.07 \\
\midrule
 Parallel MF-H&\\
\midrule
$1/8$ & 3.97847e-01 & 0.00 & 2.59003e-02 & 0.00 & 7.32330e-01 & 0.00 & 7.37024e-01 & 0.00 \\
$1/16$ & 1.01965e-01 & 1.96 & 6.26849e-03 & 2.05 & 2.70842e-01 & 1.44 & 2.72223e-01 & 1.44 \\
$1/32$ & 2.56523e-02 & 1.99 & 1.55716e-03 & 2.01 & 1.16387e-01 & 1.22 & 1.16611e-01 & 1.22 \\
$1/64$ & 6.42322e-03 & 2.00 & 3.88772e-04 & 2.00 & 5.54321e-02 & 1.07 & 5.54595e-02 & 1.07 \\
\midrule
Parallel MH-F&\\
\midrule
$1/8$ & 3.97847e-01 & 0.00 & 2.59003e-02 & 0.00 & 7.32703e-01 & 0.00 & 7.36652e-01 & 0.00 \\
$1/16$ & 1.01965e-01 & 1.96 & 6.26849e-03 & 2.05 & 2.70913e-01 & 1.44 & 2.72152e-01 & 1.44 \\
$1/32$ & 2.56523e-02 & 1.99 & 1.55716e-03 & 2.01 & 1.16394e-01 & 1.22 & 1.16605e-01 & 1.22 \\
$1/64$ & 6.42322e-03 & 2.00 & 3.88772e-04 & 2.00 & 5.54305e-02 & 1.07 & 5.54611e-02 & 1.07 \\
\midrule
Parallel M-F-H&\\
\midrule
$1/8$ & 3.97846e-01 & 0.00 & 2.58917e-02 & 0.00 & 7.31122e-01 & 0.00 & 7.35429e-01 & 0.00 \\
$1/16$ & 1.01965e-01 & 1.96 & 6.26606e-03 & 2.05 & 2.70616e-01 & 1.43 & 2.71920e-01 & 1.44 \\
$1/32$ & 2.56523e-02 & 1.99 & 1.55662e-03 & 2.01 & 1.16359e-01 & 1.22 & 1.16576e-01 & 1.22 \\
$1/64$ & 6.42322e-03 & 2.00 & 3.88727e-04 & 2.00 & 5.54315e-02 & 1.07 & 5.54605e-02 & 1.07 \\
\bottomrule
\label{tab: K Theta 0.1 nu 0.3}
\end{tabular*}
\end{table} 

 \begin{table}\tiny  
\caption{Example 1: The errors and convergence orders with decreasing mesh sizes of algorithms. $ \nu=0.4999999$. }
\centering
\begin{tabular*}{0.9\textwidth}{@{\extracolsep{\fill}} l ccccccccc @{\extracolsep{\fill}}}
\toprule
  $h$&$\Vert\bm u\Vert_{H^1}$ &rate &$\Vert\xi\Vert_{L^2}$ &rate &$\Vert p\Vert_{H^1}$ &rate&$\Vert T\Vert_{H^1}$ &rate \\ 
\midrule
 MFH &\\
\midrule
$1/8$ & 3.95271e-01 & 0.00 & 4.13935e-02 & 0.00 & 4.32622e-01 & 0.00 & 4.32630e-01 & 0.00 \\
$1/16$ & 1.01202e-01 & 1.97 & 9.72317e-03 & 2.09 & 2.17480e-01 & 0.99 & 2.17481e-01 & 0.99 \\
$1/32$ & 2.54560e-02 & 1.99 & 2.40345e-03 & 2.02 & 1.08887e-01 & 1.00 & 1.08887e-01 & 1.00 \\
$1/64$ & 6.37382e-03 & 2.00 & 5.99499e-04 & 2.00 & 5.44618e-02 & 1.00 & 5.44618e-02 & 1.00 \\
\midrule
Parallel M-FH &\\
\midrule
$1/8$ & 3.95271e-01 & 0.00 & 4.13935e-02 & 0.00 & 4.32622e-01 & 0.00 & 4.32630e-01 & 0.00 \\
$1/16$ & 1.01202e-01 & 1.97 & 9.72317e-03 & 2.09 & 2.17480e-01 & 0.99 & 2.17481e-01 & 0.99 \\
$1/32$ & 2.54560e-02 & 1.99 & 2.40345e-03 & 2.02 & 1.08887e-01 & 1.00 & 1.08887e-01 & 1.00 \\
$1/64$ & 6.37382e-03 & 2.00 & 5.99499e-04 & 2.00 & 5.44618e-02 & 1.00 & 5.44618e-02 & 1.00 \\
\midrule
 Parallel MF-H&\\
\midrule
$1/8$ & 3.95271e-01 & 0.00 & 4.13935e-02 & 0.00 & 4.32623e-01 & 0.00 & 4.32630e-01 & 0.00 \\
$1/16$ & 1.01202e-01 & 1.97 & 9.72317e-03 & 2.09 & 2.17480e-01 & 0.99 & 2.17481e-01 & 0.99 \\
$1/32$ & 2.54560e-02 & 1.99 & 2.40345e-03 & 2.02 & 1.08887e-01 & 1.00 & 1.08887e-01 & 1.00 \\
$1/64$ & 6.37382e-03 & 2.00 & 5.99494e-04 & 2.00 & 5.44618e-02 & 1.00 & 5.44618e-02 & 1.00 \\
\midrule
Parallel MH-F&\\
\midrule
$1/8$ & 3.95271e-01 & 0.00 & 4.13935e-02 & 0.00 & 4.32623e-01 & 0.00 & 4.32630e-01 & 0.00 \\
$1/16$ & 1.01202e-01 & 1.97 & 9.72317e-03 & 2.09 & 2.17480e-01 & 0.99 & 2.17481e-01 & 0.99 \\
$1/32$ & 2.54560e-02 & 1.99 & 2.40345e-03 & 2.02 & 1.08887e-01 & 1.00 & 1.08887e-01 & 1.00 \\
$1/64$ & 6.37382e-03 & 2.00 & 5.99494e-04 & 2.00 & 5.44618e-02 & 1.00 & 5.44618e-02 & 1.00 \\
\midrule
Parallel M-F-H&\\
\midrule
$1/8$ & 3.95271e-01 & 0.00 & 4.13935e-02 & 0.00 & 4.32623e-01 & 0.00 & 4.32630e-01 & 0.00 \\
$1/16$ & 1.01202e-01 & 1.97 & 9.72317e-03 & 2.09 & 2.17480e-01 & 0.99 & 2.17481e-01 & 0.99 \\
$1/32$ & 2.54560e-02 & 1.99 & 2.40345e-03 & 2.02 & 1.08887e-01 & 1.00 & 1.08887e-01 & 1.00 \\
$1/64$ & 6.37382e-03 & 2.00 & 5.99494e-04 & 2.00 & 5.44618e-02 & 1.00 & 5.44618e-02 & 1.00 \\
\bottomrule
\label{tab: nu 0.49999}
\end{tabular*}
\end{table}

\subsection{Pressure and temperature numerical instability}\label{subsec: Pressure and temperature oscillations}

In this subsection, we investigate the numerical instability of pressure and temperature under the limiting condition where $a_0$, $b_0$, and $c_0$ are sufficiently small. For the sake of simplicity, we further set $a_0 = b_0 = c_0 = 0$, and assign the permeability-like tensor $\bm{K}$ and thermal conductivity tensor $\bm{\Theta}$ the small values $\bm{K} = \bm{\Theta} = 10^{-14}\bm{I}$.
We only perform a single time step integration, i.e., $\Delta t = \tau = 0.001$.
Figure \ref{fig:six-images} assesses the accuracy of the proposed algorithms for the three-field model in this limiting configuration, while all remaining parameters are kept identical to the baseline setup. The stabilization parameters are set to $\rho_p =\rho_T =\rho_\xi=\tfrac{1}{\lambda}$.
Since temperature shares the identical real-valued solution structure with pressure, only the pressure field results are presented here to avoid redundant content; the temperature counterpart exhibits exactly the same behavior.
As can be observed, the raw three-field model without stabilization produces spurious numerical instability in the pressure field, whereas the presented stabilization algorithms effectively eliminate these unphysical artifacts.

\begin{figure}[htbp]
  \centering
  \begin{subfigure}[b]{0.3\textwidth}
    \centering
    \includegraphics[width=\linewidth]{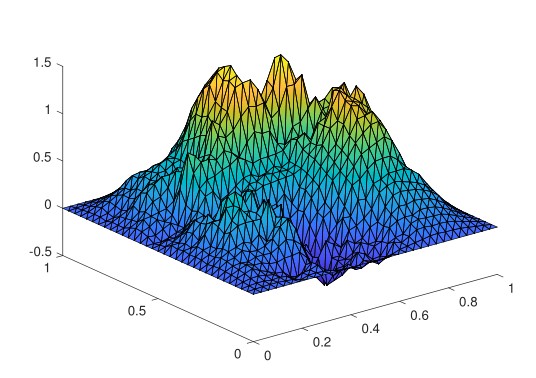}
    \caption{Three field model}
    \label{fig:img1}
  \end{subfigure}\hfill
  \begin{subfigure}[b]{0.3\textwidth}
    \centering
    \includegraphics[width=\linewidth]{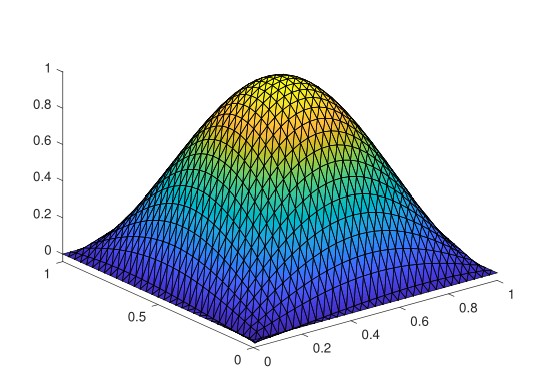}
    \caption{MFH}
    \label{fig:img2}
  \end{subfigure}\hfill
  \begin{subfigure}[b]{0.3\textwidth}
    \centering
    \includegraphics[width=\linewidth]{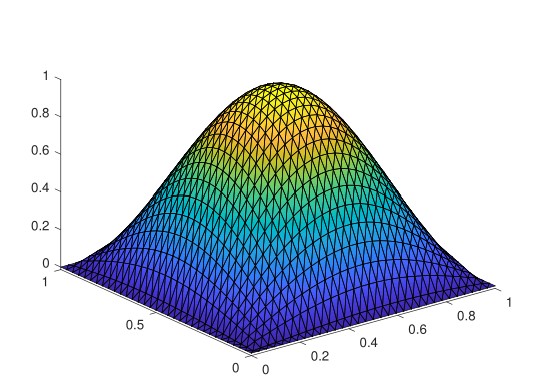}
    \caption{M-FH}
    \label{fig:img3}
  \end{subfigure}
  \vspace{1em} %
  \begin{subfigure}[b]{0.3\textwidth}
    \centering
    \includegraphics[width=\linewidth]{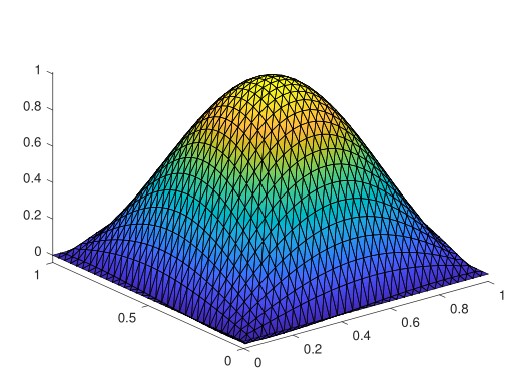}
    \caption{MF-H}
    \label{fig:img4}
  \end{subfigure}\hfill
  \begin{subfigure}[b]{0.3\textwidth}
    \centering
    \includegraphics[width=\linewidth]{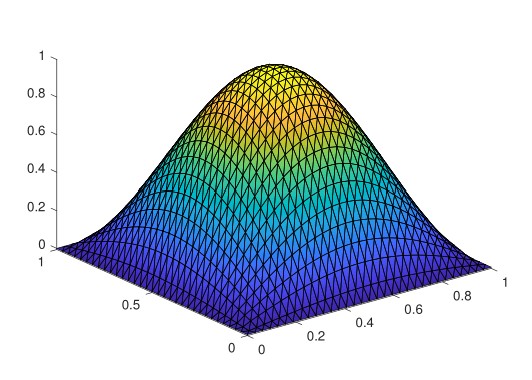}
    \caption{MH-F}
    \label{fig:img5}
  \end{subfigure}\hfill
  \begin{subfigure}[b]{0.3\textwidth}
    \centering
    \includegraphics[width=\linewidth]{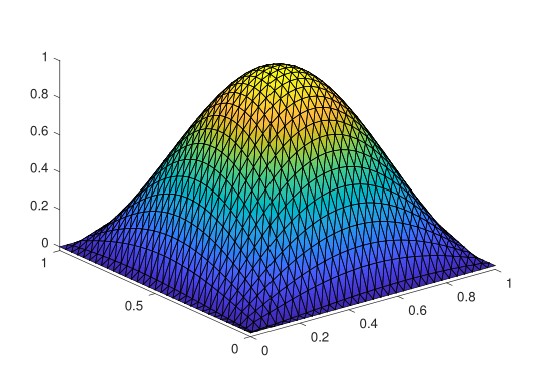}
    \caption{M-F-H}
    \label{fig:img6}
  \end{subfigure}
  \caption{Pressure of six algorithms}
  \label{fig:six-images}
\end{figure}

\subsection{Convection dominate and strong coupling effects}\label{subsec: Strong nonlinear and coupling effects}
We also examine two challenging regimes: a convection-dominated case ($c_f=100$) and a strong coupling case ($a_0=c_0=1$, $b_0=0$, $\alpha=\beta=2$). The results are given in Table \ref{tab: strong nonlinear} and Table \ref{tab: strong coupling}, respectively. In each case, all iterative algorithms maintain the expected convergence rates, and the decoupled methods produce errors comparable to the coupled algorithm. We apply stabilization parameters $\rho_p = \rho_T = \rho_\xi=10$ to all algorithms.

\begin{table}[htbp]  \tiny  
\caption{Example 1: The errors and convergence orders with decreasing mesh sizes of
 algorithms. $c_f=100$.  }
\centering
\begin{tabular*}{0.9\textwidth}{@{\extracolsep{\fill}} l ccccccccc @{\extracolsep{\fill}}}
\toprule
  $h$&$\Vert\bm u\Vert_{H^1}$ &rate &$\Vert\xi\Vert_{L^2}$ &rate &$\Vert p\Vert_{H^1}$ &rate&$\Vert T\Vert_{H^1}$ &rate \\ 
\midrule
 MFH &\\
\midrule
$1/8$ & 3.97715e-01 & 0.00 & 2.44757e-02 & 0.00 & 4.38722e-01 & 0.00 & 4.43847e-01 & 0.00 \\
$1/16$ & 1.01927e-01 & 1.96 & 5.85009e-03 & 2.06 & 2.18782e-01 & 1.00 & 2.21021e-01 & 1.01 \\
$1/32$ & 2.56426e-02 & 1.99 & 1.44500e-03 & 2.02 & 1.09155e-01 & 1.00 & 1.10055e-01 & 1.01 \\
$1/64$ & 6.42079e-03 & 2.00 & 3.63288e-04 & 1.99 & 5.45455e-02 & 1.00 & 5.47404e-02 & 1.01 \\
\midrule
Parallel M-FH &\\
\midrule
$1/8$ & 3.97715e-01 & 0.00 & 2.44722e-02 & 0.00 & 4.37664e-01 & 0.00 & 4.42086e-01 & 0.00 \\
$1/16$ & 1.01927e-01 & 1.96 & 5.84910e-03 & 2.06 & 2.18579e-01 & 1.00 & 2.20466e-01 & 1.00 \\
$1/32$ & 2.56425e-02 & 1.99 & 1.44482e-03 & 2.02 & 1.09132e-01 & 1.00 & 1.09915e-01 & 1.00 \\
$1/64$ & 6.42080e-03 & 2.00 & 3.63344e-04 & 1.99 & 5.45505e-02 & 1.00 & 5.47572e-02 & 1.01 \\
\midrule
 Parallel MF-H&\\
\midrule
$1/8$ & 3.97715e-01 & 0.00 & 2.44657e-02 & 0.00 & 4.35984e-01 & 0.00 & 4.38306e-01 & 0.00 \\
$1/16$ & 1.01927e-01 & 1.96 & 5.84759e-03 & 2.06 & 2.18251e-01 & 1.00 & 2.19302e-01 & 1.00 \\
$1/32$ & 2.56426e-02 & 1.99 & 1.44506e-03 & 2.02 & 1.09093e-01 & 1.00 & 1.09621e-01 & 1.00 \\
$1/64$ & 6.42090e-03 & 2.00 & 3.64566e-04 & 1.99 & 5.45619e-02 & 1.00 & 5.48077e-02 & 1.00 \\
\midrule
Parallel MH-F&\\
\midrule
$1/8$ & 3.97715e-01 & 0.00 & 2.44658e-02 & 0.00 & 4.35293e-01 & 0.00 & 4.39160e-01 & 0.00 \\
$1/16$ & 1.01927e-01 & 1.96 & 5.84760e-03 & 2.06 & 2.18115e-01 & 1.00 & 2.19498e-01 & 1.00 \\
$1/32$ & 2.56426e-02 & 1.99 & 1.44506e-03 & 2.02 & 1.09076e-01 & 1.00 & 1.09649e-01 & 1.00 \\
$1/64$ & 6.42090e-03 & 2.00 & 3.64573e-04 & 1.99 & 5.45677e-02 & 1.00 & 5.47948e-02 & 1.00 \\
\midrule
Parallel M-F-H&\\
\midrule
$1/8$ & 3.97715e-01 & 0.00 & 2.44719e-02 & 0.00 & 4.37572e-01 & 0.00 & 4.41961e-01 & 0.00 \\
$1/16$ & 1.01927e-01 & 1.96 & 5.84901e-03 & 2.06 & 2.18557e-01 & 1.00 & 2.20425e-01 & 1.00 \\
$1/32$ & 2.56425e-02 & 1.99 & 1.44480e-03 & 2.02 & 1.09128e-01 & 1.00 & 1.09902e-01 & 1.00 \\
$1/64$ & 6.42079e-03 & 2.00 & 3.63344e-04 & 1.99 & 5.45499e-02 & 1.00 & 5.47543e-02 & 1.01 \\
\bottomrule
\label{tab: strong nonlinear}
\end{tabular*}
\end{table}

\begin{table}[htbp]  \tiny  
\caption{Example 1: The errors and convergence orders with decreasing mesh sizes of
 algorithms. $a_0=c_0=1$, $b_0=0$, $\alpha=\beta=2$.  }
\centering
\begin{tabular*}{0.9\textwidth}{@{\extracolsep{\fill}} l ccccccccc @{\extracolsep{\fill}}}
\toprule
  $h$&$\Vert\bm u\Vert_{H^1}$ &rate &$\Vert\xi\Vert_{L^2}$ &rate &$\Vert p\Vert_{H^1}$ &rate&$\Vert T\Vert_{H^1}$ &rate \\ 
\midrule
 MFH &\\
\midrule
$1/8$ & 4.07461e-01 & 0.00 & 8.44351e-02 & 0.00 & 5.10198e-01 & 0.00 & 5.10488e-01 & 0.00 \\
$1/16$ & 1.04736e-01 & 1.96 & 2.27568e-02 & 1.89 & 2.36045e-01 & 1.11 & 2.36189e-01 & 1.11 \\
$1/32$ & 2.63910e-02 & 1.99 & 5.83932e-03 & 1.96 & 1.13145e-01 & 1.06 & 1.13207e-01 & 1.06 \\
$1/64$ & 6.61273e-03 & 2.00 & 1.41517e-03 & 2.04 & 5.52967e-02 & 1.03 & 5.53132e-02 & 1.03 \\
\midrule
Parallel M-FH &\\
\midrule
$1/8$ & 4.05615e-01 & 0.00 & 8.04397e-02 & 0.00 & 4.95547e-01 & 0.00 & 4.95784e-01 & 0.00 \\
$1/16$ & 1.04163e-01 & 1.96 & 2.11017e-02 & 1.93 & 2.32158e-01 & 1.09 & 2.32273e-01 & 1.09 \\
$1/32$ & 2.62394e-02 & 1.99 & 5.19991e-03 & 2.02 & 1.12197e-01 & 1.05 & 1.12246e-01 & 1.05 \\
$1/64$ & 6.58217e-03 & 2.00 & 1.35645e-03 & 1.94 & 5.52115e-02 & 1.02 & 5.52257e-02 & 1.02 \\
\midrule
 Parallel MF-H&\\
\midrule
$1/8$ & 4.07087e-01 & 0.00 & 8.37686e-02 & 0.00 & 5.54473e-01 & 0.00 & 4.69875e-01 & 0.00 \\
$1/16$ & 1.04618e-01 & 1.96 & 2.24065e-02 & 1.90 & 2.46674e-01 & 1.17 & 2.26488e-01 & 1.05 \\
$1/32$ & 2.63619e-02 & 1.99 & 5.67514e-03 & 1.98 & 1.15114e-01 & 1.10 & 1.11266e-01 & 1.03 \\
$1/64$ & 6.60998e-03 & 2.00 & 1.38778e-03 & 2.03 & 5.54131e-02 & 1.05 & 5.52036e-02 & 1.01 \\
\midrule
Parallel MH-F&\\
\midrule
$1/8$ & 4.07087e-01 & 0.00 & 8.37686e-02 & 0.00 & 4.69736e-01 & 0.00 & 5.54700e-01 & 0.00 \\
$1/16$ & 1.04618e-01 & 1.96 & 2.24065e-02 & 1.90 & 2.26418e-01 & 1.05 & 2.46790e-01 & 1.17 \\
$1/32$ & 2.63619e-02 & 1.99 & 5.67514e-03 & 1.98 & 1.11232e-01 & 1.03 & 1.15165e-01 & 1.10 \\
$1/64$ & 6.60998e-03 & 2.00 & 1.38778e-03 & 2.03 & 5.51907e-02 & 1.01 & 5.54260e-02 & 1.06 \\
\midrule
Parallel M-F-H&\\
\midrule
$1/8$ & 4.07463e-01 & 0.00 & 8.38327e-02 & 0.00 & 5.10340e-01 & 0.00 & 5.10533e-01 & 0.00 \\
$1/16$ & 1.04707e-01 & 1.96 & 2.23605e-02 & 1.91 & 2.35810e-01 & 1.11 & 2.35906e-01 & 1.11 \\
$1/32$ & 2.63473e-02 & 1.99 & 5.50948e-03 & 2.02 & 1.12837e-01 & 1.06 & 1.12879e-01 & 1.06 \\
$1/64$ & 6.56650e-03 & 2.00 & 1.31040e-03 & 2.07 & 5.51509e-02 & 1.03 & 5.51619e-02 & 1.03 \\
\bottomrule
\label{tab: strong coupling}
\end{tabular*}
\end{table}

\subsection{Computational efficiency}\label{sec: com eff}

As noted earlier, a key advantage of the proposed methods is their parallelizability. This subsection evaluates the computational performance of the proposed algorithms. The terminal time is set to $\tau = 0.001$, and the time step is set to $\Delta t= 0.001$ as well, so only one time step is needed to reach the terminal time. The number of iterations is set to $10$. The physical parameters are maintained identical to those adopted in the baseline. 
In the experiments, the M‌FH algorithm is executed on a single CPU core by a sequential program. For the remaining parallel algorithms, subproblems are explicitly dispatched to two or three independent CPU cores, and they are executed concurrently.
 All these numerical experiments were conducted on a workstation equipped with ‌an Intel(R) Core(TM) i5-8250U CPU running Ubuntu 24.04 LTS.
Table \ref{tab: time compare 1} reports the CPU times for all algorithms. MFH takes the longest CPU time, while parallel M-F-H and parallel M-FH take the shortest.

\begin{table}[htbp] \tiny  
\caption{Comparison of CPU time. $\tau=\Delta t=0.001$. }
\centering
\begin{tabular}{cccccccc}
\toprule
  &  $\Vert\bm u\Vert_{H^1}$  &$\Vert\xi\Vert_{L^2}$  &$\Vert p\Vert_{H^1}$ &$\Vert T\Vert_{H^1}$ & CPU time(s)\\ 
\midrule
MFH            &6.42327e-03&3.89307e-04&5.56057e-02&5.56479e-02&69.86\\    
parallel M-FH  &6.42327e-03&3.89313e-04&5.56056e-02&5.56478e-02&19.97\\
parallel MF-H  &6.42327e-03&3.89268e-04&5.56055e-02&5.56468e-02&30.15\\
parallel MH-F  &6.42327e-03&3.89268e-04&5.56046e-02&5.56476e-02&27.15\\ 
parallel M-F-H &6.42326e-03&3.89111e-04&5.56041e-02&5.56462e-02&18.40\\
\bottomrule
\label{tab: time compare 1}
\end{tabular}
\end{table}

\subsection*{Acknowledgments} 

\subsection*{Funding}
Project 12601707 supported by National Natural Science Foundation of China.


\bibliographystyle{plain}
\bibliography{ref}



\end{document}